\documentclass[reqno,reqno]{amsart}
\usepackage[latin9]{inputenc}
\usepackage{color}
\usepackage{refstyle}
\usepackage{mathrsfs}
\usepackage{amstext}
\usepackage{amsthm}
\usepackage{amssymb}
\usepackage{graphicx}
\usepackage[unicode=true,pdfusetitle,
 bookmarks=true,bookmarksnumbered=false,bookmarksopen=false,
 breaklinks=false,pdfborder={0 0 1},backref=false,colorlinks=true]
 {hyperref}
\hypersetup{
 pdfborderstyle=,citecolor=blue,linkcolor=blue,linktocpage=true,urlcolor=black}

\makeatletter

\AtBeginDocument{\providecommand\subsecref[1]{\ref{subsec:#1}}}
\AtBeginDocument{\providecommand\secref[1]{\ref{sec:#1}}}
\AtBeginDocument{\providecommand\propref[1]{\ref{prop:#1}}}
\AtBeginDocument{\providecommand\corref[1]{\ref{cor:#1}}}
\AtBeginDocument{\providecommand\lemref[1]{\ref{lem:#1}}}
\AtBeginDocument{\providecommand\thmref[1]{\ref{thm:#1}}}
\AtBeginDocument{\providecommand\figref[1]{\ref{fig:#1}}}
\AtBeginDocument{\providecommand\defref[1]{\ref{def:#1}}}
\AtBeginDocument{\providecommand\remref[1]{\ref{rem:#1}}}
\RS@ifundefined{subsecref}
  {\newref{subsec}{name = \RSsectxt}}
  {}
\RS@ifundefined{thmref}
  {\def\RSthmtxt{theorem~}\newref{thm}{name = \RSthmtxt}}
  {}
\RS@ifundefined{lemref}
  {\def\RSlemtxt{lemma~}\newref{lem}{name = \RSlemtxt}}
  {}

\numberwithin{equation}{section}
\numberwithin{figure}{section}
\theoremstyle{plain}
\newtheorem{thm}{\protect\theoremname}[section]
\theoremstyle{remark}
\newtheorem{rem}[thm]{\protect\remarkname}
\theoremstyle{definition}
\newtheorem{defn}[thm]{\protect\definitionname}
\theoremstyle{remark}
\newtheorem{note}[thm]{\protect\notename}
\theoremstyle{plain}
\newtheorem{lem}[thm]{\protect\lemmaname}
\theoremstyle{plain}
\newtheorem{cor}[thm]{\protect\corollaryname}
\theoremstyle{definition}
\newtheorem{example}[thm]{\protect\examplename}
\theoremstyle{plain}
\newtheorem{prop}[thm]{\protect\propositionname}
\theoremstyle{remark}
\newtheorem*{acknowledgement*}{\protect\acknowledgementname}

\AtBeginDocument{\providecommand\defref[1]{\ref{def:#1}}}\AtBeginDocument{}\AtBeginDocument{\providecommand\remref[1]{\ref{rem:#1}}}\AtBeginDocument{}\allowdisplaybreaks
\usepackage{needspace}
\usepackage{enumitem}
\usepackage{bbm}
\usepackage{url}

\newref{lem}{refcmd={Lemma \ref{#1}}}
\newref{thm}{refcmd={Theorem \ref{#1}}}
\newref{cor}{refcmd={Corollary \ref{#1}}}
\newref{sec}{refcmd={Section \ref{#1}}}
\newref{sub}{refcmd={Section \ref{#1}}}
\newref{subsec}{refcmd={Section \ref{#1}}}
\newref{chap}{refcmd={Chapter \ref{#1}}}
\newref{prop}{refcmd={Proposition \ref{#1}}}
\newref{exa}{refcmd={Example \ref{#1}}}
\newref{tab}{refcmd={Table \ref{#1}}}
\newref{rem}{refcmd={Remark \ref{#1}}}
\newref{def}{refcmd={Definition \ref{#1}}}
\newref{fig}{refcmd={Figure \ref{#1}}}
\newref{que}{refcmd={Question \ref{#1}}}

\DeclareMathOperator*{\argmax}{\mathit{argmax}}

\setlist[enumerate]{itemsep=5pt,topsep=3pt}
\setlist[enumerate,1]{label=\textup{(}\roman*\textup{)},ref=\roman*}
\setlist[enumerate,2]{label=(\alph*),ref=\theenumi \alph*}

\usepackage{tikz}
\usetikzlibrary{positioning, arrows.meta}

\tikzset{%
    myarrow/.style = {-Stealth, shorten >=5pt}
}

\theoremstyle{definition}

\providecommand{\acknowledgementname}{Acknowledgement}

\providecommand{\corollaryname}{Corollary}
\providecommand{\definitionname}{Definition}
\providecommand{\examplename}{Example}
\providecommand{\lemmaname}{Lemma}
\providecommand{\propositionname}{Proposition}
\providecommand{\remarkname}{Remark}
\providecommand{\theoremname}{Theorem}

\let\oldnumberline\numberline%
\renewcommand{\numberline}{\figurename~\oldnumberline}

\@ifundefined{showcaptionsetup}{}{%
 \PassOptionsToPackage{caption=false}{subfig}}
\usepackage{subfig}
\makeatother

\providecommand{\acknowledgementname}{Acknowledgement}
\providecommand{\corollaryname}{Corollary}
\providecommand{\definitionname}{Definition}
\providecommand{\examplename}{Example}
\providecommand{\lemmaname}{Lemma}
\providecommand{\notename}{Note}
\providecommand{\propositionname}{Proposition}
\providecommand{\remarkname}{Remark}
\providecommand{\theoremname}{Theorem}

\begin{document}
\title[Dimension Reduction and Kernel Principal Component Analysis]{An Infinite dimensional Analysis of Kernel Principal Components }
\author{Palle E.T. Jorgensen}
\address{(Palle E.T. Jorgensen) Department of Mathematics, The University of
Iowa, Iowa City, IA 52242-1419, U.S.A.}
\email{palle-jorgensen@uiowa.edu}
\author{Sooran Kang}
\address{(Sooran Kang) College of General Education, Choongang University,
Seoul, Korea}
\email{sooran09@cau.ac.kr}
\author{Myung-Sin Song}
\address{(Myung-Sin Song) Department of Mathematics and Statistics, Southern
Illinois University Edwardsville, Edwardsville, IL 62026, USA}
\email{msong@siue.edu}
\author{James Tian}
\address{(James F. Tian) Mathematical Reviews, 416 4th Street Ann Arbor, MI
48103-4816, U.S.A.}
\email{jft@ams.org}
\begin{abstract}
We study non-linear data-dimension reduction. We are motivated by
the classical linear framework of Principal Component Analysis. In
nonlinear case, we introduce instead a new kernel-Principal Component
Analysis, manifold and feature space transforms. Our results extend
earlier work for probabilistic Karhunen-Loève transforms on compression
of wavelet images. Our object is algorithms for optimization, selection
of efficient bases, or components, which serve to minimize entropy
and error; and hence to improve digital representation of images,
and hence of optimal storage, and transmission. We prove several new
theorems for data-dimension reduction. 
\end{abstract}

\keywords{dimension reduction, principal component analysis, eigenvalues, optimization,
Karhunen-Loève transform, optimal storage, transmission, wavelet image
compression, frames, entropy encoding, algorithms, Hilbert space.}
\subjclass[2000]{62H25, 34L16, 65K10, 65T60, 42C15, 47B06, 47B32, 65D15, 47A70, 37M25,
42C40.}

\maketitle
\tableofcontents{}

\renewcommand{\listfigurename}{Data and digital image illustrations}

\listoffigures

\section{\label{sec:intro}Introduction}

Recently a number of new features of principal component analysis
(PCA) have lead to exciting and new improved dimension reduction (DR).
See e.g., \cite{Belkin03,MR3857315,MR3883202,4106847,MR2247587,MR3005666,MR2177937,MR3854652,MR3534893}.
In general DR refers to the process of reducing the number of random
variables under consideration in such areas as machine learning, statistics,
and information theory. Within machine learning, it involves both
the steps of feature selection and feature extraction. In the present
paper, we shall consider linear as well as non-linear data models.
The linear case arises naturally in principal component analysis (PCA).
See \cite{So07,JoSo07a}. Here one starts with a linear mapping of
the given data into a suitable lower-dimensional space. However, this
must be done in such a way that the variance of the data in the low-dimensional
representation is maximized. As for the variance, we study both covariance,
and the correlation matrix for the underlying data. The eigenvectors
that correspond to the largest eigenvalues (the principal components)
are then used in a construction of a large fraction of the initial
variance, i.e., that which corresponds to the original data. The first
few eigenvectors are typically interpreted in terms of the large-scale
physical behavior of a particular system; and will retain the most
important variance features. 

In nonlinear settings, principal component analysis can still be adapted,
but now by means of suitable kernel tricks. In some applications,
instead of starting from a fixed kernel, the optimization will instead
try to learn, or adapt, for example with the use of semidefinite programs
\cite{MR2741485,MR1976484}. The most prominent such a technique is
known as maximum variance unfolding (MVU) \cite{MR3104494,MR2278445}.

The aim of this paper is to set the stage for further developments
of the linking of the mathematics of kernel-PCA (KPCA) with real-life
non-linear data-sets, coming from a host of new developments in the
current applied literature. A key focus for our work is a study of
two complementary directions: on one side, the mathematics of kernel
theory, and on the other, empirical (here, numerical experiments starting
with databases from real-world situations). For the purpose, we develop
(Section \ref{sec:kpca} ) particular kernel tools with view to KPCA
via statistical simulation, as well as corresponding analyses of kernel
based choices of feature spaces and feature maps.

In recent years, the subject of kernel-principal component analysis
(KPCA), and its applications, has been extensively studied, and made
progress in diverse directions. There is a variety of themes covered
in the literature dealing with applications of kernel-PCA tools is
vast, each of some relevance to our present theme. A short user guide:
the paper \cite{MR3857315} deals with sensitivity issues, \cite{MR3883202}
direct integration algorithms, \cite{4106847} image processing, \cite{MR2247587,MR3005666,MR3854652}
pattern recognition and machine learning, \cite{MR2177937} statistical
properties, and finally \cite{MR3534893} on selection of efficient
parameterizations. In addition to earlier papers by the co-authors
\cite{So07,JoSo07a}, we also include here a partial list of other
relevant and current citations; see e.g., the papers mentioned above,
\cite{MR3780557,MR3820672,MR3911884,MR1720704,MR3857315,MR3878657,MR3913046,MR3934645,MR3850675,MR3900802,MR3922239}.

\subsection{Reproducing kernels and operator theory}

In our present paper we stress, and motivate, the most general framework
for choices of positive definite kernels, their corresponding RKHSs,
and families of compatible measure spaces. These considerations encompass
more general contexts for Principal Component Analysis, going beyond
classical cases of kernel-based PCA in the literature. In particular,
this very general and non-linear framework for PCA data-dimension
reduction goes beyond existing, more classical, considerations admitting
easy realizations in $\mathbb{R}^{n}$, classical kernels, and classical
Fourier tools. Indeed, there are existing parallel considerations,
in the literature involving only kernel analyses for such special
cases.

However, existing non-linear realizations of big and un-structured
data motivate choices from a much wider family of RKHSs, and measure
spaces. This further entails infinite-dimensional stochastic analysis
frameworks which do not admit realizations in $\mathbb{R}^{n}$, nor
classical PDE/Fourier techniques. Examples of uses, in the literature,
of general RKHSs in stochastic approaches include online learning
with Markov sampling, data mining (e.g., in finance, intelligence,
hidden data, tele-communication, energy, unstructured information
integrated with traditional structured data), as well as a diverse
host of different investigations, see e.g., \cite{MR3104494,MR4188895,MR3231624,MR2247587,MR1976484,MR2239907,MR1864085,MR3912285,MR3507188,MR3736758,Jorgensen2018,MR2482257,MR2558684,MR2488871,MR2177937,MR3642406,aggarwal2018machine,hogan2020web}.
Moreover \textquotedblleft big data\textquotedblright{} further dictate
diverse, large, infinite-dimensional, and extremely non-linear data-sets
which do not lend themselves to feature analysis with choices of anyone
of the simple kernels with easy realizations in $\mathbb{R}^{n}$.
Examples of this includes many public available datasets on Global
Health Observatory, U.S. Census Bureau, NIH Data Sharing Repositories,
Kaggle, etc. Such cases of course dictate new and diverse choices
from a much wider family of kernels and RKHSs. Encoding of meaningful
programs for corresponding non-linear data-dimension reduction algorithms,
while feasible, is a gigantic task, going far beyond the scope of
our present paper. It is the subject of future projects involving
the present authors, as well as others, see e.g., the papers cited
above.

In addition, we call attention to the 2018 paper \cite{8248802} by
Yamada et al. It is motivated by existing and new biological data.
For the purpose, the authors introduce a novel Hilbert-Schmidt independence
criterion. They apply it to specific high dimensional, and large-scale
datasets, as arising in ultra-high dimensional biological data. The
authors obtain classification into phenotypes, module expression in
human prostate cancer patients, and detection of enzymes in protein
structures. The 2019 paper \cite{9002955} by Chen et al deals with
a new class of kernel based nonlinear clustering algorithms. They
study locality structures in machine learning. Also motivated by specific
large-scale datasets is the 2020 paper \cite{9248120}. This paper
introduces new hyper-parameter initialization tools for kernel-based
regression analysis. Yet different but related kernel based technology,
and specific data sets, are further covered in the papers \cite{1556085},
\cite{10.1145/1015330.1015345}, \cite{914517}, and \cite{doi:10.1148/ryai.2021210011}.

Motivated by related applications, we focus on non-traditional positive
definite kernels and the associated RKHSs, in particular, when the
family of \emph{feature spaces} may be chosen in the form $L^{2}\left(\mu\right)$.
This raises the question of which measures $\mu$ are right for a
particular kernel $K$ and its RKHS. 

In the use of probabilistic tools in big data, such as Monte Carlo
simulation (see e.g., (\ref{eq:br6}) below), a main issue is reduction
of the required sample size; \emph{and still maintain} (i) acceptable
small noise in output, (ii) small error term estimates, as well as
(iii) reduction of computation time. These issues are especially relevant
for modern image processing. As a consequence, more traditional kernel
tools must therefore be adapted. Both our present paper, and an earlier
one \cite{MR4295177}, serve to do this by shifting the focus in the
design of optimization algorithms to that of making choices of \emph{measurable
partitions}, as opposed to working directly with points.

For instance, when $K=K^{\left(\mu\right)}$ is defined on a $\sigma$-finite
measure space $\left(V,\mathscr{B},\mu\right)$, where $K^{\left(\mu\right)}\left(A,B\right)=\mu\left(A\cap B\right)$,
for all $A,B$ in $\mathscr{B}_{fin}=\left\{ A\in\mathscr{B}:\mu\left(A\right)<\infty\right\} $,
the associated RKHS $\mathscr{H}(K^{\left(\mu\right)})$ is shown
to be a Hilbert space of signed finite measures $F$ on $\left(V,\mathscr{B}\right)$,
such that $F\ll\mu$ and $\left\Vert F\right\Vert _{\mathscr{H}(K^{\left(\mu\right)})}=\left\Vert dF/d\mu\right\Vert _{L^{2}\left(\mu\right)}$;
see \cite[Thms 3.2,  3.4]{MR4295177}. Moreover, in this measurable
setting, the classical Karhunen-Loève decomposition takes the following
form instead: 
\begin{equation}
X_{A}=\sum_{n=0}^{\infty}\left(\int_{A}f_{n}d\mu\right)Z_{n},\quad A\in\mathscr{B}_{fin}\label{eq:br6}
\end{equation}
where $X_{A}$ is a centered Gaussian process indexed by $\mathscr{B}_{fin}$,
$\left\{ f_{n}\right\} _{n\in\mathbb{N}_{0}}$ is an orthonormal basis
(ONB) in $L^{2}\left(\mu\right)$, and $\left\{ Z_{n}\right\} _{n\in\mathbb{N}_{0}}$
is an i.i.d. $N\left(0,1\right)$ system. These new kernels encompass
realistic cases when random measurements (random variables) might
not be performed at points, but instead at sample sets that are selected
from suitable choices of sigma-algebras. For more details on the more
general context of Parseval frames in the measure category, see \subsecref{gf}
below, and \cite[Cor 3.13 and sect.  5]{MR4295177}. 

\textbf{Main results.} The presentation of our main results is organized
as follows: \secref{kpca} covers multiple interrelated tools, playing
a key role in our main results. The first of these tools is our use
of \emph{kernels} (also called \emph{reproducing kernels}) and their
associated Hilbert spaces, called \emph{reproducing kernel Hilbert
spaces} (RKHSs.) The use of these RKHSs yields powerful tools, and
they allow explicit realizations of a variety of embeddings of \emph{non-linear
structures} into linear ones (Hilbert spaces). This is useful in turn
for our solution to \emph{optimization questions}, see \propref{sel}
and \corref{pce}. Our results are tested on datasets, see Figures
\ref{fig:bc} and \ref{fig:bc-1}. 

\textbf{Organization.} The goal of this paper is to extend our previous
results on Karhunen-Loève transform to a nonlinear setting by means
of kernel-principal component analysis (KPCA). This paper is organized
as follows: 

In \secref{kpca}, we show our main results using KPCA on nonlinear
data. See \lemref{rk1}. The latter is for rank-1 projections, but
is then extended to PCA selection, and algorithms, in subsequent results.
Our focus is the non-linear case. Indeed, the focus in \secref{kpca}
is nonlinear data dimension reduction, and the corresponding kernels,
the core of our paper. In particular, our \thmref{rcpt} deals with
kernel PCA for nonlinear data dimension reduction (see Examples \ref{exa:sc}
and \ref{exa:dr}). In the remaining part of \secref{kpca}, we turn
to the case of Gaussian kernels.

KPCA have found wide applications. The focus of our current paper
is to present a framework of operator in suitable Hilbert space, and
an associated spectral theory. Even though the applications we include
here are presented in finite dimensional setting, most of our results
extend to infinite dimensional spaces as well. Nonetheless, for use
in recursion schemes, the finite-dimensional case is most relevant.

Our main results deal with algorithms for optimization in maximal
variance, and dimension reduction-problems, from PCA. The points where
our results go beyond the earlier literature includes the following
list of four closely interrelated items: (i) Our design and use of
kernel tricks (specific reproducing kernel Hilbert spaces (RKHSs)
integrated into new designs of PCA-tools for image-analysis, thus
serving as practical tools in dimension-reduction algorithms; (ii)
Our combination of RKHS-tools with spectral theory and Hilbert space
frame-estimates; --- this serves in turn to make more precise both
specific numerical PCA-algorithms, and their error estimates; (iii)
Our use of RKHSs in creating explicit and practical embeddings of
non-linear data into suitable linear spaces (feature spaces); ---
applications to optimization (\subsecref{ao}); and (iv) A new dynamical
PCA-analysis (\subsecref{dyn})

Our results in \secref{kpca} include \lemref{rk1}, \thmref{rcpt},
and \corref{pce}; each of which are formulated and proved in a general
setting of kernel analysis; hence in a non-linear framework of feature
selection. 

Our main purpose is the identification and proof of rigorous results
from infinite-dimensional analysis which are needed in the solution
of optimization problems in the general framework of Principal Component
Analysis with an emphasis on the role of PCA-algorithms for selection
of features in infinite dimensions. Main points in the paper include
our infinite-dimensional optimization results, Theorems \ref{thm:opF},
\ref{thm:opF2}, \ref{thm:rcpt}, Remark \ref{rem:In-the-above},
and Corollaries \ref{cor:pce}, and \ref{cor:The-system-of}.

\section{\label{sec:kpca}Kernel PCA}

\textbf{Comparison of KPCA to standard PCA} (brief sketch.) The standard
PCA always finds linear principal components. It serves to represent
a given large data set into a suitable choice of principal components
for lower dimension. But PCA will fail to find good representative
directions when applied to most \emph{non-linear data} of interest.
For non-linear and more tricky data sets, one turns instead to Kernel
PCA (KPCA). It has the following advantages and features (not a complete
list): (i) KPCA will perform PCA but in a new space. (ii) It uses
diverse kernel tricks in order to still find principal components,
but they will be in a different space (typically in a higher dimensional
space.) (iii) KPCA will use Hilbert space geometry in order to still
find new directions based on families of kernel-matrices. Via geometric
algorithms, KPCA will then extract corresponding eigenvalues (corresponding
to a suitable number of observations.) However, (iv) the computational
complexity for KPCA, dictated by extraction of principal components,
will take more time, as compared to standard PCA. The list of papers
on KPCA and applications to diverse sets of non-linear data is long;
some recent ones are \cite{Harkat:2020tx,rs11101219,MR2249872,MR2819026,10.1145/1015330.1015417,naik2017advances,kpcai}.

In this section, we make precise, and show that the choice of Gaussian
kernel is \textquotedblleft optimal\textquotedblright{} for KPCA in
capturing maximal variance (i.e., maximal variability of data). There
is an extensive literature on PCA and KPCA. Our approach and emphasis
are different. For instance, the paper \cite{rosipal2001kernel} deals
with certain Least Squares Regression problems with the use of a particular
(small) class of Reproducing tools involving Kernel Hilbert Space.
While there are connections to our present themes, our present results
are much more general, offering in particular specific spectral theoretic,
and stochastic tools, penalty terms, especially relevant for digital
images, and for dynamical PCA. The focus of \cite{10.1162/089976698300017467}
is the design of certain Nonlinear Component Analysis as a particular
class of Kernel Eigenvalue Problems. By comparison, our present results
and applications have a different focus, use different tools both
as regards to infinite-dimensional analysis, optimization, and Gaussian
fields; as we outline below. This paper \cite{MR290013} deals with
a particular class of splines (numerical analysis). This in turn entails
a class of kernels. But the focus of \cite{MR290013} is quite different
from ours. The paper \cite{WU1997165} deals with matrix algorithms
close in spirit to our present paper, but our focus and applications
are global in nature, and are very different as we outline below.
The book \cite{cristianini_shawe-taylor_2000} deals with kernel methods
and some of their applications to support vector machines. This is
an exciting application of feature selections, and learning algorithms,
but the aim and the tools developed in this book are quite different
from our present focus.

This allows us to obtain extensions of some PCA-results from \cite{JoSo07a}.
We shall focus here on their use in kernel PCA (KPCA). 

PCA is used in data dimension reduction on linear case. However, this
cannot be done on nonlinear case and thus kernel principal component
analysis (KPCA) is used for nonlinear dimension reduction. See, e.g.,
\cite{MR2558684,MR2488871,MR2327597,MR2228737,MR1968413,MR1864085}.

Standard PCA is effective at identifying linear subspaces carrying
the greatest variance in a data set. However, this method is not able
to detect nonlinear submanifolds. A popular technique to tackle the
latter case is kernel PCA. It first maps data into a higher dimensional
space $\mathcal{H}\left(K\right)$, and performs PCA there. Here $\mathcal{H}\left(K\right)$
is the reproducing kernel Hilbert space (RKHS) associated with a given
positive definite kernel $K$. (Recall that an RKHS is a Hilbert space
$\mathcal{H}$ of functions defined on a set $X\left(\neq\emptyset\right)$,
such that the evaluation functionals $\mathcal{H}\ni f\rightarrow f\left(x\right)\in\mathbb{C}$
are bounded for all $x\in X$. See also \defref{RKHS} below.) The
mapping in this context presumably sends a nonlinear submanifold in
the input space to a linear subspace in $\mathcal{H}\left(K\right)$.
For example, in classification problems, a kernel is usually chosen
so that the mapped data can be separated by a linear decision boundary
in $\mathcal{H}\left(K\right)$ (see \figref{bc}). 
\begin{rem}
It would be intriguing to compare Smale's Dimension Reduction algorithm
from \cite{MR2558684} with ours. The two approaches are along a different
lines of development.

The approach in Belkin's paper \cite{Belkin03} is popular in current
Machine Learning research. Both our results and those of Belkin et
al aim for dimension reduction algorithms. Other methods exist, which
constitute variants of KPCA, but with different choices of kernels,
and with the Laplacian eigenmap (LE) as one of them. (See also \cite{MR3912285,MR3935831,MR3935948,MR3934524}.)
For recent developments on graph Laplacians, and Perron-Frobenius
eigenfunctions as principal components, we refer to e.g., \cite{MR1913212}. 

As applications of the theory presented in the last section, \ref{sec:kpca},
we include a discussion of new non-linear, and real-life data/simulated
examples. Moreover, inside the sections, we offer explanations for
our use of RKHS-theory, as applied in new KPCA based Gaussian process
simulations. In particular, we outline how they serve to produce optimal
choices for detection of maximal variance.
\end{rem}

The \emph{optimization} issues in the present section refer to details
in subsections \ref{subsec:of} and \ref{subsec:dual}; see especially
\lemref{rk1}, \remref{ork} and \thmref{rcpt}, all adapted to
the PCA issues at hand. Our notation for solutions to optimization
problems is ``\emph{argmax}'' (standard terminology in the optimization
literature.) The role of \emph{kernels} enters via associated choices
of\emph{ feature maps} $\Phi$, see (\ref{eq:du1}), (\ref{eq:dp2}),
(\ref{eq:d1}), (\ref{eq:A3}), (\ref{eq:A4}), and subsection \ref{subsec:FS}.
Furthermore, our choices of \emph{frames}, global frame-operators,
and their adjoints (\propref{sel}), then allow a\emph{ duality approach}
which in turn helps us gain insight into algorithms for the optimizers. 

By a \textbf{feature} we mean an individual measurable property or
characteristic. One chooses informative, discriminating and \emph{independent}
features to be used in algorithmic constructions in machine learning,
in pattern recognition, and regression. Vector spaces associated with
particular features are called the \textbf{feature spaces}. In KPCA,
the choices of \emph{feature spaces} come about from specified \emph{kernels}
(assumed positive definite, p.d.), see \defref{RKHS} below. If $K$
is a fixed p.d. kernel defined on general non-linear sets of data,
a natural choice of feature space is then in the form of a \emph{Hilbert
space} is called \emph{the reproducing kernel Hilbert space} (RKHS),
written $\mathscr{H}\left(K\right)$, see \remref{rk} below.
\begin{defn}
\label{def:RKHS}Let $S$ be a set. A positive definite (p.d.) kernel
on $S$ is a function $K:S\times S\rightarrow\mathbb{C}$, such that
\begin{equation}
\sum_{i,j=1}^{N}\overline{c_{i}}c_{j}K\left(x_{i},x_{j}\right)\geq0\label{eq:pd1}
\end{equation}
for all $\left\{ x_{i}\right\} _{i=1}^{N}\subset S$, $\left\{ c_{i}\right\} _{i=1}^{N}\subset\mathbb{C}$,
and $N\in\mathbb{N}$.

Given a p.d. kernel as in (\ref{eq:pd1}), there exists a reproducing
kernel Hilbert space (RKHS) $\mathcal{H}\left(K\right)$ and a mapping
$\Phi:S\rightarrow\mathcal{H}\left(K\right)$ such that 
\begin{equation}
K\left(x,y\right)=\left\langle \Phi\left(x\right),\Phi\left(y\right)\right\rangle _{\mathcal{H}\left(K\right)}.\label{eq:A2}
\end{equation}
The function $\Phi$ in (\ref{eq:A2}) is called a \emph{feature map}
for the problem.

Moreover, the following reproducing property holds: 
\begin{equation}
f\left(x\right)=\left\langle K_{x},f\right\rangle _{\mathcal{H}\left(K\right)},
\end{equation}
for all $f\in\mathcal{H}\left(K\right)$, and $x\in S$. 
\end{defn}

While the theory of \emph{positive definite kernels} and their RKHSs
have served for a long time as powerful tools in diverse areas of
pure and applied mathematics, the particular applications of them
we need here are of a more recent vintage; for background literature,
see e.g., \cite{MR3402823,MR2239907,MR3450534,MR3736758,Jorgensen2018}.
The theory of \emph{frames} (non-orthogonal expansions) also plays
an essential role in our present approach to PCA, and the list of
relevant papers for background includes \cite{MR2367342,MR2193805,MR3441732,MR3085820}.

\textbf{Feature maps.} In the subsequent discussion, reference to
\emph{kernel} will always mean positive definite kernel, see (\ref{eq:pd1}).
For each choice of kernel, defined on data sets (typically non-linear
configurations), there will then be associated \emph{feature maps}.
Generally, a feature map $\Phi$ is simply a \emph{function} which
maps sets of data configurations (non-linear) into some choice of
\emph{feature space}, a Hilbert space, say $\mathscr{L}$. Typically
the \emph{inner product} in $\mathscr{L}$ will serve to model correlation,
or regression numbers, or other measurements corresponding to other
features of interest. Given a kernel $K$, a natural choice of feature
space $\mathscr{L}$ is the RKHS defined directly from $K$, see (\ref{eq:A2}),
but there are other possibilities. In a different choice for $\Phi$,
the range of $\Phi$, and the inner product on the right-hand side
in (\ref{eq:A2}) might be different. The main logic in the use of
kernels in \emph{machine learning} is that it yields representations
of learning algorithm for data sets that will then become more amenable
to regression analysis, detection, learning, and to classification.
Analysis is aided with the use of powerful tools from the theory of
linear operators in Hilbert space; see e.g., \cite{MR2367342,Jorgensen2018,MR2327597}.
We stress that the phrase \emph{feature map} is broad, and that a
wide variety of functions $\Phi$ may serve as feature maps. But the
main use of them relates to suitable choices of kernels. \emph{Support
Vector Machines} (and other kernel-based methods) make use both implicit
and explicit feature maps. This leads to remapping of data and will
allow non-linearly separable data sets to get linear representations.
The feature space leads to effective separation of data via suitable
choices of hyperplanes in higher dimension. But reaching these dimensions
might be computationally expensive because bad choices of feature
mappings might require many computations. 
\begin{rem}
\label{rem:rk}$\mathcal{H}\left(K\right)$ may be chosen as the Hilbert
completion of 
\begin{equation}
\text{span}\left\{ K_{x}:=K\left(\cdot,x\right)\right\} \label{eq:a4}
\end{equation}
with respect to the $\mathcal{H}\left(K\right)$-inner product 
\begin{equation}
\left\langle \sum c_{i}K_{x_{i}},\sum d_{j}K_{x_{j}}\right\rangle _{\mathcal{H}\left(K\right)}:=\sum\overline{c_{i}}d_{j}K\left(x_{i},x_{j}\right).\label{eq:a5}
\end{equation}

Initially the LHS in formula (\ref{eq:a5}) only refers to finite
linear combinations. Hence, the vector space (\ref{eq:a4}) becomes
a pre-Hilbert space. (By pre-Hilbert space we mean an inner product
space that is not complete.) The RKHS $\mathcal{H}(K)$ itself then
results from the standard Hilbert completion. It is this Hilbert space
we will use in our subsequent study of optimization problems, and
in our KPCA-dimension reduction. Sections \ref{subsec:ao}--\ref{subsec:of}
deal with separate issues of kernel-optimization. Before turning to
these, however, we will first introduce a setting of Hilbert-Schmidt
operators. This will play a crucial role in the formulation of our
main result, \thmref{rcpt} in \subsecref{of}. 
\end{rem}

Recall that a data set $\left\{ x_{j}\right\} _{j=1}^{n}$, $x_{j}\in\mathbb{C}^{m}$,
may be viewed as an $m\times n$ matrix $X$, where $x_{j}$ is the
$j^{th}$ column vector. Here, $m$ is the number of features, and
$n$ the number of sample points. Set 
\begin{equation}
\left\Vert X\right\Vert _{HS}^{2}=\sum\left|x_{ij}\right|^{2}=tr\left(X^{*}X\right),\label{eq:hsn}
\end{equation}
where $\left\Vert \cdot\right\Vert _{HS}$ in (\ref{eq:hsn}) denotes
the Hilbert-Schmidt norm. 
\begin{rem}
Let $\mathcal{H}$ be a Hilbert space, and let $HS\left(\mathcal{H}\right)$
be the Hilbert-Schmidt operators $\mathcal{H}\xrightarrow{\;X\;}\mathcal{H}$
with inner product 
\begin{equation}
\left\langle X,Y\right\rangle =tr\left(X^{*}Y\right).
\end{equation}
Then the two Hilbert spaces $HS\left(\mathcal{H}\right)$, and $\mathcal{H}\otimes\overline{\mathcal{H}}$
(tensor-product), are naturally isometrically isomorphic via 
\begin{equation}
HS\left(\mathcal{H}\right)\ni\left|u\left\rangle \right\langle v\right|\longrightarrow u\otimes\overline{v}\in\mathcal{H}\otimes\overline{\mathcal{H}},\label{eq:hs2}
\end{equation}
Indeed, 
\[
\left\Vert \left|u\left\rangle \right\langle v\right|\right\Vert _{HS}^{2}=\left\Vert u\right\Vert _{\mathcal{H}}^{2}\left\Vert v\right\Vert _{\mathcal{H}}^{2},
\]
and the assertion follows from isometric extension of (\ref{eq:hs2}). 
\end{rem}

\subsection{\label{subsec:ao}Application to Optimization}

One of the more recent applications of kernels and the associated
reproducing kernel Hilbert spaces (RKHS) is to optimization, also
called kernel-optimization. See \cite{MR3803845,MR2933765}. In the
context of machine learning, it refers to training-data and feature
spaces. In the context of numerical analysis, a popular version of
the method is used to produce splines from sample points; and to create
best spline-fits. In statistics, there are analogous optimization
problems going by the names ``least-square fitting,'' and ``maximum-likelihood''
estimation. In the latter instance, the object to be determined is
a suitable probability distribution which makes ``most likely''
the occurrence of some data which arises from experiments, or from
testing.

What these methods have in common is a minimization (or a max problem)
involving a ``quadratic'' expression $Q$ with two terms. The first
in $Q$ measures a suitable $L^{2}\left(\mu\right)$-square applied
to a difference of a measurement and a ``best fit.'' The latter
will then to be chosen from a number of suitable reproducing kernel
Hilbert spaces (RKHS). The choice of kernel and RKHS will serve to
select desirable features. So we will minimize a quantity $Q$ which
is the sum of two terms as follows: (i) a $L^{2}$-square applied
to a difference, and (ii) a penalty term which is a RKHS norm-squared.
(See eq. (\ref{eq:ao2}).) In the application to determination of
splines, the penalty term may be a suitable Sobolev normed-square;
i.e., $L^{2}$ norm-squared applied to a chosen number of derivatives.
Hence non-differentiable choices will be ``penalized.''

In all of the cases, discussed above, there will be a good choice
of (i) and (ii), and we show that there is then an explicit formula
for the optimal solution; see eq (\ref{eq:ao5}) in \thmref{opF}
below. 

Let $X$ be a set, and let $K:X\times X\longrightarrow\mathbb{C}$
be a positive definite (p.d.) kernel. Let $\mathcal{H}\left(K\right)$
be the corresponding reproducing kernel Hilbert space (RKHS). Let
$\mathscr{B}$ be a sigma-algebra of subsets of $X$, and let $\mu$
be a positive measure on the corresponding measure space $\left(X,\mathscr{B}\right)$.
We assume that $\mu$ is sigma-finite. We shall further assume that
the associated operator $T$ given by 
\begin{equation}
\mathcal{H}\left(K\right)\ni f\xrightarrow{\;T\;}\left(f\left(x\right)\right)_{x\in X}\in L^{2}\left(\mu\right)\label{eq:ao1}
\end{equation}
is densely defined and closable.

Fix $\beta>0$, and $\psi\in L^{2}\left(\mu\right)$, and set 
\begin{equation}
Q_{\psi,\beta}\left(f\right)=\left\Vert \psi-Tf\right\Vert _{L^{2}\left(\mu\right)}^{2}+\beta\left\Vert f\right\Vert _{\mathcal{H}\left(K\right)}^{2}\label{eq:ao2}
\end{equation}
defined for $f\in\mathcal{H}\left(K\right)$ , or in the dense subspace
$dom\left(T\right)$ where $T$ is the operator in (\ref{eq:ao1}).
Let 
\begin{equation}
L^{2}\left(\mu\right)\xrightarrow{\;T^{*}\;}\mathcal{H}\left(K\right)\label{eq:ao3}
\end{equation}
be the corresponding adjoint operator, i.e., 
\begin{equation}
\left\langle F,T^{*}\psi\right\rangle _{\mathcal{H}\left(K\right)}=\left\langle Tf,\psi\right\rangle _{L^{2}\left(\mu\right)}=\int_{X}\overline{f\left(s\right)}\psi\left(s\right)d\mu\left(s\right).\label{eq:ao4}
\end{equation}

Our present RKHS/$L^{2}\left(\mu\right)$ framework is close to that
of \cite{MR3560890}. But we have included the results we need in
our present framework, and for our purpose. 
\begin{thm}
\label{thm:opF}Let $K$, $\mu$, $\psi$, $\beta$ be as specified
above; then the optimization problem 
\begin{equation}
\inf_{f\in\mathcal{H}\left(K\right)}Q_{\psi,\beta}\left(f\right)\label{eq:a0o04a}
\end{equation}
has a unique solution $F$ in $\mathcal{H}\left(K\right)$, it is
\begin{equation}
F=\left(\beta I+T^{*}T\right)^{-1}T^{*}\psi\label{eq:ao5}
\end{equation}
where the operator $T$ and $T^{*}$ are as specified in (\ref{eq:ao1})-(\ref{eq:ao4}). 
\end{thm}

\begin{proof}
We fix $F$, and assign $f_{\varepsilon}:=F+\varepsilon h$ where
$h$ varies in the dense domain $dom\left(T\right)$ from (\ref{eq:ao1}).
For the derivative $\frac{d}{d\varepsilon}\big|_{\varepsilon=0}$
we then have: 
\[
\frac{d}{d\varepsilon}\big|_{\varepsilon=0}Q_{\psi,\beta}\left(f_{\varepsilon}\right)=2\Re\left\langle h,\left(\beta I+T^{*}T\right)F-T^{*}\psi\right\rangle _{\mathcal{H}\left(K\right)}=0
\]
for all $h$ in a dense subspace in $\mathcal{H}\left(K\right)$.
The desired conclusion follows. Note that convexity of the function
$Q_{\psi,\beta}$ makes these conditions sufficient as well. 
\end{proof}
As already mentioned, the general conclusion in the present \thmref{opF}
connects to both many uses of \emph{kernel} techniques in classical
analysis, as well as to new and exciting applications. In the second
group, we mention the following approach to \emph{Neural Networks}
\cite{MR4296249}, leading in turn to encoding of \emph{deep regularization}
into training of inner layers which make up specific Neural Network
(NN) constructions, with the use of \emph{kernel flows}. For details,
we refer to \cite{MR4296249} and the cited literature for details.
But here we only stress that, as a step in particular NN-designs,
one makes use of an recursive iteration of substitution into a fixed
\emph{kernel} (the kernel resulting from the substitution iterations
is called a \emph{warped kernel}); and then, at each step, there will
be an RKHS-optimization of a term that arises as a special case of
(\ref{eq:a0o04a}), see \cite[eq (1.2)]{MR4296249}.
\begin{flushleft}
\textbf{Least-square Optimization} 
\par\end{flushleft}

To help readers put our present results on feature space and RKHS
penalty-term into context, we now specialize the optimization formula
from \thmref{opF} to the problem of minimize a ``quadratic'' quantity
$Q$. It is still the sum of two individual terms: (i) a $L^{2}$-square
applied to a difference, and (ii) a penalty term which is the RKHS
norm-squared. But the least-square term in (i) will simply be a sum
of a finite number of squares of differences; hence ``least-squares.''
As an application, we then get an easy formula (\thmref{opF2}) for
the optimal solution.

Let $K$ be a positive definite kernel on $X\times X$ where $X$
is an arbitrary set, and let $\mathcal{H}\left(K\right)$ be the corresponding
reproducing kernel Hilbert space (RKHS). Let $m\in\mathbb{N}$, and
consider sample points:

$\left\{ t_{j}\right\} _{j=1}^{m}$ as a finite subset in $X$, and

$\left\{ y_{i}\right\} _{i=1}^{m}$ as a finite subset in $\mathbb{R}$,
or equivalently, a point in $\mathbb{R}^{m}$.

Fix $\beta>0$, and consider $Q=Q_{\left(\beta,t,y\right)}$, defined
by 
\begin{equation}
Q\left(f\right)=\sum_{i=1}^{m}\underset{\text{least square}}{\left|f\left(t_{i}\right)-y_{i}\right|^{2}}+\beta\underset{\text{penalty form}}{\left\Vert f\right\Vert _{\mathcal{H}\left(K\right)}^{2}},\quad f\in\mathcal{H}\left(K\right).\label{eq:a23}
\end{equation}
We introduce the associated dual pair of operators as follows: 
\begin{equation}
\begin{split}T:\mathcal{H}\left(K\right) & \longrightarrow\mathbb{R}^{m}\simeq l_{m}^{2},\;\mbox{and}\\
T^{*}:l_{m}^{2} & \longrightarrow\mathcal{H}\left(K\right)
\end{split}
\label{eq:a24}
\end{equation}
where 
\begin{align}
Tf & =\left(f\left(t_{i}\right)\right)_{i=1}^{m},\quad f\in\mathcal{H}\left(K\right);\:\mbox{and}\label{eq:a25}\\
T^{*}y & =\sum_{i=1}^{m}y_{i}K\left(\cdot,t_{i}\right)\in\mathcal{H}\left(K\right),\label{eq:a26}
\end{align}
for all $\vec{y}=\left(y_{i}\right)\in\mathbb{R}^{m}$.

Note that the duality then takes the following form: 
\begin{equation}
\left\langle T^{*}y,f\right\rangle _{\mathcal{H}\left(K\right)}=\left\langle y,Tf\right\rangle _{l_{m}^{2}},\quad\forall f\in\mathcal{H}\left(K\right),\:\forall y\in l_{m}^{2};\label{eq:a27}
\end{equation}
consistent with (\ref{eq:ao4}).

Applying \thmref{opF} to the counting measure 
\[
\mu=\sum_{i=1}^{m}\delta_{t_{i}}=\delta_{\left\{ t_{i}\right\} }
\]
for the set of sample points $\left\{ t_{i}\right\} _{i=1}^{m}$,
we get the two formulas: 
\begin{align}
T^{*}Tf & =\sum_{i=1}^{m}f\left(t_{i}\right)K\left(\cdot,t_{i}\right)=\sum_{i=1}^{m}f\left(t_{i}\right)K_{t_{i}},\:\mbox{and }\label{eq:a28}\\
TT^{*}\vec{y} & =K_{m}\vec{y}\label{eq:a29}
\end{align}
where $K_{m}$ denotes the $m\times m$ matrix 
\begin{equation}
K_{m}=\left(K\left(t_{i},t_{j}\right)\right)_{i,j=1}^{m}=\begin{pmatrix}K\left(t_{1},t_{1}\right) & \cdots & \cdots & K\left(t_{1},t_{m}\right)\\
K\left(t_{2},t_{1}\right) & \cdots & \cdots & K\left(t_{2},t_{m}\right)\\
\vdots\\
K\left(t_{m},t_{1}\right) &  &  & K\left(t_{m},t_{m}\right)
\end{pmatrix}.\label{eq:a30}
\end{equation}

\begin{thm}
\label{thm:opF2}Let $K$, $X$, $\left\{ t_{i}\right\} _{i=1}^{m}$,
and $\left\{ y_{i}\right\} _{i=1}^{m}$ be as above, and let $K_{m}$
be the induced sample matrix (\ref{eq:a30}). Fix $\beta>0$; consider
the optimization problem with 
\begin{equation}
Q_{\beta,\left\{ t_{i}\right\} ,\left\{ y_{i}\right\} }\left(f\right)=\sum_{i=1}^{m}\left|y_{i}-f\left(t_{i}\right)\right|^{2}+\beta\left\Vert f\right\Vert _{\mathcal{H}\left(K\right)}^{2},\quad f\in\mathcal{H}\left(K\right).\label{eq:a31}
\end{equation}
Then the unique solution to (\ref{eq:a31}) is given by 
\begin{equation}
F\left(\cdot\right)=\sum_{i=1}^{m}\left(K_{m}+\beta I_{m}\right)_{i}^{-1}K\left(\cdot,t_{i}\right)\:\mbox{on \ensuremath{X};}\label{eq:a32}
\end{equation}
i.e., $F=\arg\min Q$ on $\mathcal{H}\left(K\right)$. 
\end{thm}

\begin{proof}
From \thmref{opF}, we get that the unique solution $F\in\mathcal{H}\left(K\right)$
is given by: 
\[
\beta F+T^{*}TF=T^{*}y,
\]
and by (\ref{eq:a28})-(\ref{eq:a29}), we further get 
\begin{equation}
\beta F\left(\cdot\right)=\sum_{i=1}^{m}\left(y_{i}-F\left(t_{i}\right)\right)K\left(\cdot,t_{i}\right)\label{eq:a33}
\end{equation}
where the dot $\cdot\phantom{}$ refers to a free variable in $X$.
An evaluation of (\ref{eq:a33}) on the sample points yields: 
\begin{equation}
\beta\vec{F}=K_{m}\left(\vec{y}-\vec{F}\right)\label{eq:a34}
\end{equation}
where $\vec{F}:=\left(F\left(t_{i}\right)\right)_{i=1}^{m}$, and
$\vec{y}=\left(y_{i}\right)_{i=1}^{m}$. Hence 
\begin{equation}
\vec{F}=\left(\beta I_{m}+K_{m}\right)^{-1}K_{m}\vec{y}.\label{eq:a35}
\end{equation}
Now substitute (\ref{eq:a35}) into (\ref{eq:a34}), and the desired
conclusion in the theorem follows. We used the matrix identity 
\[
I_{m}-\left(\beta I_{m}+K_{m}\right)^{-1}K_{m}=\beta\left(\beta I_{m}+K_{m}\right)^{-1}.
\]
\end{proof}

\subsection{\label{subsec:gf}The Case of Gaussian Fields}

For a number of applications, it will be convenient to consider general
stochastic processes $\left(X_{s}\right)$ indexed by $s\in S$, where
$S$ is merely a \emph{set}; so not \emph{a priori} equipped with
any additional structure. Consideration of stochastic processes will
always assume some fixed probability space, $\left(\Omega,\mathscr{A},\mathbb{P}\right)$
where $\Omega$ is a set of sample points; $\mathscr{A}$ is a $\sigma$-algebra
of events, fixed at the outset; and $\mathbb{P}$ is a probability
measure defined on $\mathscr{A}$. A given process $\left(X_{s}\right)_{s\in S}$
is then said to be \emph{Gaussian} and centered iff (Def.) for all
choice of finite subsets of $S$ ($s_{1},s_{2},\cdots,s_{N}$), then
the system of random variables $\left\{ X_{s_{i}}\right\} _{i=1}^{N}$
is jointly Gaussian, i.e., the joint distribution of $\left\{ X_{s_{i}}\right\} _{i=1}^{N}$
on $\mathbb{R}^{N}$ is the Gaussian $g_{N}\left(x_{1},\cdots,x_{N}\right)$
which has mean zero, and covariance matrix 
\begin{equation}
K_{ij}^{\left(N\right)}:=\left(\mathbb{E}\left(X_{s_{i}}X_{s_{j}}\right)\right)_{i,j=1}^{N};
\end{equation}
so for $x=\left(x_{1},\cdots,x_{N}\right)\in\mathbb{R}^{N}$, 
\begin{equation}
g_{N}\left(x\right)=\left(2\pi\right)^{-N/2}\det\big(K^{\left(N\right)}\big)^{-1/2}e^{-\frac{1}{2}x^{T}K^{\left(N\right)^{-1}}x.}
\end{equation}
If $A_{N}\subset\mathbb{R}^{N}$ is a Borel set, then 
\begin{equation}
\mathbb{P}\left(\left(X_{s_{1}},\cdots,X_{s_{N}}\right)\in A_{N}\right)=\int_{A_{N}}g_{N}\left(x\right)d^{N}x
\end{equation}
holds. Note we consider the joint distributions for all finite subsets
of $S$.

Let $S$ be a set, and let $\left\{ X_{s}\right\} _{s\in S}$ be a
Gaussian process with $\mathbb{E}\left(X_{s}\right)=0$, $\forall s\in S$;
and with 
\begin{equation}
\mathbb{E}\left(\overline{X}_{s}X_{t}\right):=K_{X}\left(s,t\right)\label{eq:g1}
\end{equation}
as its covariance kernel. Finally, let $\mathcal{H}\left(K\right)$
be the corresponding RKHS.

\emph{Then the following general results hold} (see e.g., \cite{MR2793121,MR2966130,MR3231624,MR3402823,MR3450534,MR3507188,MR3721329,MR3687240,MR3736758,Jorgensen2018}): 
\begin{enumerate}
\item \label{enu:g1}Every positive definite kernel $S\times S\xrightarrow{\;K\;}\mathbb{C}$
arises as in (\ref{eq:g1}) from some Gaussian process $\left\{ X_{s}\right\} _{s\in S}$. 
\item \label{enu:g2}Assume $\mathcal{H}\left(K\right)$ is separable; then
we have a representation $\left\{ f_{n}\right\} _{n\in\mathbb{N}}$
for a system of functions $f_{n}:S\rightarrow\mathbb{C}$, $n\in\mathbb{N}$,
\begin{equation}
K\left(s,t\right)=\sum_{n\in\mathbb{N}}\overline{f_{n}\left(s\right)}f_{n}\left(t\right),\label{eq:g2}
\end{equation}
absolutely convergent on $S\times S$. 
\item A system $\left\{ f_{n}\right\} _{n\in\mathbb{N}}$ satisfies (\ref{enu:g2})
if and only if it forms a \emph{Parseval frame} in $\mathcal{H}\left(K\right)$. 
\item \label{enu:g4}Given (\ref{eq:g2}), then, for every sequence of independent
identically distributed (i.i.d.) Gaussian system $\left\{ Z_{n}\right\} _{n\in\mathbb{N}},$$Z_{n}\sim N\left(0,1\right)$,
i.e., each $Z_{n}$ is a standard Gaussian random variable, $\mathbb{E}\left(Z_{n}\right)=0$,
$\mathbb{E}\left(Z_{n}Z_{m}\right)=\delta_{n,m}$; the representation
\begin{equation}
X_{s}\left(\cdot\right)=\sum_{n\in\mathbb{N}}f_{n}\left(s\right)Z_{n}\left(\cdot\right)\label{eq:fd1}
\end{equation}
is valid in $L^{2}$ of the underlying probability space, and $\mathbb{E}\left(X_{s}X_{t}\right)=K\left(s,t\right)$
on $S\times S$.
\end{enumerate}
\begin{note}
When a fixed Gaussian process $(X_{s})$ is given, then the associated
decomposition (\ref{eq:fd1}) is called a Karhunen-Loève (KL) transform
for $X_{s}$. The conclusion from (\ref{enu:g1})--(\ref{enu:g4})
above is that there is a direct connection between the two KL transforms
the relatively better known KL-transforms for positive definite kernels
(\cite{JoSo07a} Theorem 4.15).

The object in principal component analysis (PCA) is to find optimal
representations; and to select from them the ``leading terms'',
the principal components.
\end{note}

\begin{rem}
The present general kernel framework (RKHSs and Gaussian processes)
encompasses the special case we outlined in \cite{JoSo07a} Example
3.1. By way of comparison, note that the particular positive definite
kernel in the latter example is only a special case of the present
ones, see (\ref{eq:g2}) and (\ref{eq:fd1}). These types of kernels
are often referred to as the case of Mercer kernels; see also \cite{MR2228737,MR2327597,MR2558684,MR2488871}.
A Mercer kernel is continuous, and it defines a trace class operator,
as illustrated in the example. This latter feature in turn leads to
a well defined ``top part of the spectrum.'' And this then allows
us to select the principal components; i.e., the maximally correlated
variables. We shall show, in \thmref{rcpt} below, that there is an
alternative approach to principal components which applies to the
general class of positive definite kernels, and so goes far beyond
the case of Mercer kernels. 
\end{rem}

\subsection{\label{subsec:of}Optimization and Frames}

Fix a p.d. kernel $K$ on $S:=\mathbb{C}^{m}$, i.e., a functional
$K:\mathbb{C}^{m}\times\mathbb{C}^{m}\rightarrow\mathbb{C}$ satisfying
(\ref{eq:pd1}); and let $\mathcal{H}\left(K\right)$ be the associated
RKHS. In PCA, one solves the quadratic optimization problem: 
\begin{equation}
\argmax\left\{ \left\Vert QX\right\Vert _{HS}^{2}\mathrel{:}\left\Vert Q\right\Vert _{HS}^{2}=k\right\} ,\label{eq:q1}
\end{equation}
where $Q$ is a selfadjoint projection.

Kernel PCA, by contrast, solves a similar problem in $\mathcal{H}\left(K\right)$:
\begin{equation}
\argmax\left\{ \left\Vert Q\Phi\left(X\right)\right\Vert _{HS}^{2}\mathrel{:}\left\Vert Q\right\Vert _{HS}^{2}=k\right\} ,\label{eq:q2}
\end{equation}
where 
\begin{equation}
\Phi\left(X\right)=\begin{bmatrix}\Phi\left(x_{1}\right) & \cdots & \Phi\left(x_{n}\right)\end{bmatrix}.\label{eq:q3}
\end{equation}
It is understood that $\left\Vert \cdot\right\Vert _{HS}$ as in (\ref{eq:q2})
refers to the Hilbert-Schmidt class in $B\left(\mathcal{H}\left(K\right)\right)$.
\begin{flushleft}
\textbf{A Finite Frame} 
\par\end{flushleft}

Let $X$, $K$, and $\Phi$ be as above. Then $\left(\Phi\left(x_{j}\right)\right)_{j=1}^{n}$
is a finite frame whose span $\mathcal{H}_{\Phi}$ is a closed subspace
in $\mathcal{H}\left(K\right)$.

Set $L:\mathcal{H}\left(K\right)\rightarrow\mathbb{C}^{n}$ by 
\begin{equation}
Lf=\sum_{j=1}^{n}\left\langle \Phi\left(x_{j}\right),f\right\rangle _{\mathcal{H}\left(K\right)}\delta_{j},\;f\in\mathcal{H}\left(K\right),\label{eq:q4}
\end{equation}
where $\delta_{j}$ is the standard ONB in $\mathbb{C}^{m}$. The
adjoint $L^{*}:\mathbb{C}^{n}\rightarrow\mathcal{H}\left(K\right)$
is given by 
\begin{equation}
L^{*}c=\sum_{j=1}^{n}\Phi\left(x_{j}\right)c_{j},\;c=\left(c_{j}\right)\in\mathbb{C}^{n}.\label{eq:q5}
\end{equation}
It follows that 
\begin{equation}
L^{*}L=\sum_{j=1}^{n}\left|\Phi\left(x_{j}\right)\left\rangle \right\langle \Phi\left(x_{j}\right)\right|=\Phi\left(X\right)\Phi\left(X\right)^{*}.\label{eq:A9}
\end{equation}

\begin{lem}
\label{lem:rk1}Let $L^{*}L$ be as in (\ref{eq:A9}), and let 
\[
G:=L^{*}L=\sum\lambda_{j}^{2}\left|\phi_{j}\left\rangle \right\langle \phi_{j}\right|
\]
be the corresponding spectral representation, with $\lambda_{1}^{2}\geq\lambda_{2}^{2}\geq\cdots$.
Then 
\begin{align}
\left|\phi_{1}\left\rangle \right\langle \phi_{1}\right| & =\argmax\left\{ \left\Vert Q_{1}\Phi\left(X\right)\right\Vert _{HS}^{2}\mathrel{:}\left\Vert Q_{1}\right\Vert _{HS}^{2}=1,\;Q_{1}=Q_{1}^{*}=Q_{1}^{2}\right\} \label{eq:ar1}\\
 & =\argmax\left\{ tr\left(Q_{1}G\right)\mathrel{:}\left\Vert Q_{1}\right\Vert _{HS}^{2}=1\right\} .\nonumber 
\end{align}
Equivalently, the best rank-1 approximation to $G$ is 
\[
\lambda_{1}^{2}\left|\phi_{1}\left\rangle \right\langle \phi_{1}\right|.
\]
\end{lem}

\begin{rem}
\label{rem:ork}Note that the conclusion of the lemma yields a solution
the optimization problem we introduced above. Indeed, in the statement
of the lemma (see (\ref{eq:ar1})) we use the standard notation \emph{argmax}
for the data which realizes a particular optimization. In the present
case, we are maximizing a certain quadratic expression over the unit-ball
in the Hilbert-Schmidt operators. The Hilbert-Schmidt norm is designated
with the subscript HS. Part of the conclusion of the lemma asserts
that the maximum, as specified in (\ref{eq:ar1}), is attained for
a definite rank-one operator. 
\end{rem}

\begin{proof}[Proof of \lemref{rk1}]
Note that 
\[
\left\Vert Q_{1}\Phi\left(X\right)\right\Vert _{HS}^{2}=tr\left(Q_{1}\Phi\left(X\right)\Phi\left(X\right)^{*}\right)=tr\left(Q_{1}G\right).
\]
Let $w$ be a unit vector in $\mathcal{H}\left(K\right)$, and set
$Q_{1}=\left|w\left\rangle \right\langle w\right|$; then 
\begin{align}
tr\left(Q_{1}G\right) & =\sum_{j}\lambda_{j}^{2}\left|\left\langle w,\phi_{j}\right\rangle \right|^{2}.\label{eq:A1}
\end{align}
Since $\sum_{j}\left|\left\langle w,\phi_{j}\right\rangle \right|^{2}=\left\Vert w\right\Vert ^{2}=1$,
the r.h.s. of (\ref{eq:A1}) is a convex combination of $\lambda_{j}^{2}$'s;
therefore, 
\[
\sum_{j}\lambda_{j}^{2}\left|\left\langle w,\phi_{j}\right\rangle \right|^{2}\leq\lambda_{1}^{2}
\]
and equality holds if and only if $\left|\left\langle w,\phi_{1}\right\rangle \right|=1$,
and $\left\langle w,\phi_{j}\right\rangle =0$, for $j>1$.

Note that, given (\ref{eq:A1}), the conclusion of the lemma follows
directly from comparison of positive series. 
\end{proof}
\lemref{rk1} can be applied inductively which yields the best rank-1
approximation at each iteration. In fact, the result holds more generally;
see \thmref{rcpt} below. 

Our present RKHS/$L^{2}(\mu)$ framework is close to that of \cite{MR3560890}.
But we have included the results we need in our present framework,
and for our purpose.
\begin{thm}
\label{thm:rcpt} (a) Let $T:\mathcal{H}\rightarrow\mathcal{H}$ be
a Hilbert Schmidt operator, and let 
\begin{equation}
TT^{*}=\sum\lambda_{j}^{2}\left|\phi_{j}\left\rangle \right\langle \phi_{j}\right|\label{eq:hs}
\end{equation}
with $\lambda_{1}^{2}\geq\lambda_{2}^{2}\geq\cdots$, and $\lambda_{j}\rightarrow0$
as $j\rightarrow\infty$. Then 
\[
\sum_{j=1}^{n}\left|\phi_{j}\left\rangle \right\langle \phi_{j}\right|=\argmax\left\{ \left\Vert Q_{n}T\right\Vert _{HS}^{2}\mathrel{:}\left\Vert Q_{n}\right\Vert _{HS}^{2}=n,\;Q_{n}=Q_{n}^{*}=Q_{n}^{2}\right\} .
\]
Note that $\left\Vert Q_{n}T\right\Vert _{HS}^{2}=tr\left(Q_{n}G\right)$,
where $G:=TT^{*}$. See \cite{JoSo07a}. (b) Recall the assumption
that $T$ is Hilbert Schmidt is equivalent to $T^{*}T$ trace class.
We then apply the Spectral Theorem to $T^{*}T$ yielding an orthonormal
basis of eigenvectors for $T^{*}T$, and a convergent representation
for the infinite sum in (\ref{eq:hs}).
\end{thm}

\begin{proof}
Let $Q_{n}=\sum_{j=1}^{n}P_{j}$, where each $P_{j}$ is rank-1, and
$P_{i}P_{j}=0$, for $i\neq j$. Then, 
\[
\left\Vert Q_{n}T\right\Vert _{HS}^{2}=\sum_{j=1}^{n}\left\Vert P_{j}T\right\Vert _{HS}^{2},\quad\left\Vert P_{j}\right\Vert _{HS}^{2}=1.
\]
Let $\mathcal{L}$ be the corresponding Lagrangian, i.e., 
\begin{align*}
\mathcal{L} & =\sum_{j=1}^{n}\left\Vert P_{j}T\right\Vert _{HS}^{2}-\sum_{j=1}^{n}\mu_{j}\left(\left\Vert P_{j}\right\Vert _{HS}^{2}-1\right)\\
 & =\sum_{j=1}^{n}\left\langle P_{j},TT^{*}P_{j}\right\rangle _{HS}-\sum_{j=1}^{n}\mu_{j}\left(\left\langle P_{j},P_{j}\right\rangle _{HS}-1\right).
\end{align*}
It follows that 
\begin{gather}
\frac{\partial\mathcal{L}}{\partial P_{k}}=2\left(TT^{*}P_{k}-\mu_{k}P_{k}\right)=0\label{eq:tm1}\\
\Updownarrow\nonumber \\
TT^{*}P_{k}=\mu_{k}P_{k},\;\mu_{k}=\lambda_{k}^{2}.\nonumber 
\end{gather}
Hence $P_{k}$ is a spectral projection of $TT^{*}$. The conclusion
of the theorem follows from this. 

Note that one may use the usual variational directional derivative
argument applied to a sesquilinear form $\left(x,x\right)\mapsto\left\langle x,Tx\right\rangle _{\mathscr{H}}$
in \emph{any} Hilbert space $\mathscr{H}$, and obtain 
\[
\frac{d}{dx}\left\langle x,Tx\right\rangle _{\mathscr{H}}=Tx+T^{*}x,\quad x\in\mathscr{H}.
\]
In particular, if $T$ is self-adjoint, then 
\[
\frac{d}{dx}\left\langle x,Tx\right\rangle _{\mathscr{H}}=2Tx.
\]
Especially, the Hilbert space in (\ref{eq:tm1}) is the space of Hilbert-Schmidt
operators. (This argument is a key step in the proof of the spectral
theorem, i.e., taking the variational directional derivative of a
bounded bilinear form in Hilbert space.)

The precise meaning of (and the justification for) the generalized
gradient assertion (\ref{eq:tm1}) is included in \lemref{gg} below.
See also the following citations of sources for the underlying operator
theory.
\end{proof}
The variational arguments in the proof of \thmref{rcpt}, used in
the present setting of Hilbert space, and the variety of projections,
can be justified with standard tools from operator theory, including
the spectral theorem; see e.g., \cite[ch.2]{MR3642406}, \cite[ch.1]{MR4274591},
\cite[ch.2]{MR1070713} and also \cite{ArKa06,MR1913212,MR2367342,MR2558684}.
We have included more details below: 
\begin{lem}
\label{lem:gg}Suppose $T$ is a Hilbert-Schmidt (HS) operator in
a separable Hilbert space $\mathscr{H}$. For the directional derivative
$\nabla_{T}:=\frac{\partial}{\partial T}$, we have 
\[
\frac{\partial}{\partial T}\left\langle x,Tx\right\rangle _{\mathscr{H}}=\left|x\left\rangle \right\langle x\right|,\quad x\in\mathscr{H}.
\]
(Here, $\left|x\left\rangle \right\langle x\right|$ is the rank-1
operator $\mathscr{H}\ni u\mapsto\left\langle x,u\right\rangle x$
using Dirac's notation.) 
\end{lem}

\begin{proof}
 Let $Q=\left|x\left\rangle \right\langle x\right|$. Note that $f_{Q}\left(T\right):=\left\langle x,Tx\right\rangle =\text{trace}\left(TQ\right)$.
In fact, if $\left\{ e_{n}\right\} _{1}^{\infty}$ is an orthonormal
basis in $\mathscr{H}$, then 
\begin{align*}
\text{trace}\left(TQ\right) & =\sum_{n=1}^{\infty}\left\langle x,e_{n}\right\rangle _{\mathscr{H}}\left\langle e_{n},Tx\right\rangle _{\mathscr{H}}\\
 & =\left\langle x,Tx\right\rangle _{\mathscr{H}},\quad\text{by the Parseval identity.}
\end{align*}
But 
\begin{equation}
\text{trace}\left(TQ\right)=\left\langle Q,T\right\rangle _{HS}\label{eq:hs1}
\end{equation}
where $\left\langle \cdot,\cdot\right\rangle _{HS}$ denotes the Hilbert-Schmidt
inner product. 

Now we use a general fact about inner products: 
\[
\left|\left\langle Q,T\right\rangle _{HS}\right|^{2}\leq\left\langle Q,Q\right\rangle _{HS}\left\langle T,T\right\rangle _{HS},\quad\forall Q,T\in\mathscr{H}S
\]
where $\mathscr{H}S=$ the Hilbert-Schmidt operators, viewed as a
Hilbert space, relative to its trace-inner product (\ref{eq:hs1}).
That is, the Schwarz inequality is applied to the $\left\langle \cdot,\cdot\right\rangle _{HS}$-inner
product. One obtains maximum when $T\sim Q$. In particular, if $Q=\left|x\left\rangle \right\langle x\right|$,
then $Q^{2}=\left\Vert x\right\Vert ^{2}Q$ and $\text{trace}\left(Q^{2}\right)=\left\Vert x\right\Vert _{\mathscr{H}}^{4}$. 

Therefore, it follows from $f_{Q}\left(T\right)=\left\langle Q,T\right\rangle _{HS}$
that
\[
\nabla_{T}f_{Q}=\left\langle Q,\cdot\right\rangle _{HS},
\]
which is identified with $Q=\left|x\left\rangle \right\langle x\right|$,
since 
\[
\left(\mathscr{H}S\right)^{*}=\left(\mathscr{H}S\right).
\]
\end{proof}
\begin{rem}
Our use of frame analysis serves as the technical tools for ``error
estimates.'' More specifically, frames are designed such that, for
a given problem involving finite-dimensional PCA computations, the
frame estimates are designed to allow estimation of the corresponding
``error terms.'' (See \cite{JoSo07a}.)

For our present considerations of optimization questions, it may be
of interest to consider a comparison between the following two settings,
one general, and the other special: On the one hand, there is a variety
of (i) general calculus of variation issues in infinite-dimensional
contexts; and on the other, (ii) particular optimization questions
in the restricted framework of specific choices of pairs of Hilbert
spaces (still infinite-dimensional). We note that the context for
(i) is wider than that of (ii). On the other hand (as outlined below),
for the specialized framework of (ii) considered here, the use of
standard Hilbert space geometry, and a little operator theory, offer
much simplification. 

\emph{Detail}: Our present focus is (ii). By contrast, in consideration
of (i) one must address separately such technical issues as (a) functional-derivatives,
and (b) existence of extremal (optimal) solutions. By contrast, the
smaller class of optimization questions from (ii) lend themselves
directly to computation in Hilbert space. Here we note that application
of Hilbert space geometry will then simplify the two issues (a) and
(b) in (i). The point we make for our present optimization question
(i.e., a specific choice of a Hilbert space context (ii)), is that
a \textquotedblleft natural\textquotedblright{} choice of \emph{a
pair of Hilbert spaces} will be a combination of the two: an appropriate
RKHS, and a compatible $L^{2}\left(\mu\right)$-Hilbert space. This
then allows for relatively simple solutions. 
\end{rem}

\subsection{\label{subsec:dual}The Dual Problem}

Fix a data set $X=\left(x_{j}\right)_{j=1}^{n}$, $x_{j}\in\mathbb{C}^{m}$.
Let $\Phi:X\rightarrow\mathcal{H}\left(K\right)$ be the feature map
in (\ref{eq:q3}), i.e., 
\begin{equation}
\Phi\left(X\right)=\begin{bmatrix}\Phi\left(x_{1}\right) & \cdots & \Phi\left(x_{n}\right)\end{bmatrix}.\label{eq:du1}
\end{equation}
Let $L$, $L^{*}$ be the analysis and synthesis operators from (\ref{eq:q4})-(\ref{eq:q5}),
and 
\begin{equation}
L^{*}L:\mathcal{H}\left(K\right)\longrightarrow\mathcal{H}\left(K\right),\quad L^{*}Lf=\sum_{j=1}^{n}\left\langle \Phi\left(x_{j}\right),f\right\rangle \Phi\left(x_{j}\right)\label{eq:dp2}
\end{equation}
be the frame operator in (\ref{eq:A9}). In particular, 
\[
L^{*}L=\Phi\left(X\right)\Phi\left(X\right)^{*}.
\]

In view of \cite{JoSo07a} and \thmref{rcpt}, the KL basis for $L^{*}L$
contains the principal directions carrying the greatest variance in
$\Phi\left(X\right)$. In applications, it is more convenient to first
find the KL basis of $LL^{*}$ instead, where 
\begin{equation}
LL^{*}=\Phi\left(X\right)^{*}\Phi\left(X\right)=\left(K\left(x_{i},x_{j}\right)\right)_{ij=1}^{n},\label{eq:d1}
\end{equation}
as an $n\times n$ matrix in $\mathbb{C}^{n}$; see (\ref{eq:A2}).
(By general theory, if $A:\mathcal{H}\rightarrow\mathcal{H}$ is a
linear operator in a Hilbert space with dense domain, then $\sigma\left(A^{*}A\right)\backslash\left\{ 0\right\} =\sigma\left(AA^{*}\right)\backslash\left\{ 0\right\} $.) 
\begin{thm}
\label{prop:sel}Set $A:\mathbb{C}^{n}\rightarrow\mathcal{H}\left(K\right)$
by 
\begin{equation}
A\delta_{j}=\Phi\left(x_{j}\right),\label{eq:A3}
\end{equation}
and extend linearly, where $\left(\delta_{j}\right)_{j=1}^{n}$ denotes
the standard basis in $\mathbb{C}^{n}$. Then the adjoint operator
$A^{*}:\mathcal{H}\left(K\right)\rightarrow\mathbb{C}^{n}$ is 
\begin{equation}
A^{*}h=\left(\left\langle \Phi\left(x_{j}\right),h\right\rangle \right)_{j=1}^{n}\in\mathbb{C}^{n}.\label{eq:A4}
\end{equation}
That is, $A=L^{*}$ and $A^{*}=L$. 
\end{thm}

\begin{proof}
Let $v\in\mathbb{C}^{n}$, and $h\in\mathcal{H}$, then 
\begin{align*}
\left\langle Av,h\right\rangle _{\mathcal{H}\left(K\right)} & =\left\langle \sum\nolimits _{j=1}^{n}v_{j}\Phi\left(x_{j}\right),h\right\rangle _{\mathcal{H}\left(K\right)}\\
 & =\sum\nolimits _{j=1}^{n}\overline{v}_{j}\left\langle \Phi\left(x_{j}\right),h\right\rangle _{\mathcal{H}\left(K\right)}=\left\langle v,\left(\left\langle \Phi\left(x_{j}\right),h\right\rangle _{\mathcal{H}\left(K\right)}\right)\right\rangle _{\mathbb{C}^{n}},
\end{align*}
and the assertions follows. 
\end{proof}
Hence $LL^{*}:\mathbb{C}^{n}\rightarrow\mathbb{C}^{n}$ is the Gramian
matrix in $\mathbb{C}^{n}$ given by 
\begin{eqnarray}
LL^{*} & = & \Phi\left(X\right)^{*}\Phi\left(X\right)\nonumber \\
 & = & \left(\left\langle \Phi\left(x_{i}\right),\Phi\left(x_{j}\right)\right\rangle _{\mathcal{H}\left(K\right)}\right)_{i,j=1}^{n}\nonumber \\
 & \underset{\text{by \ensuremath{\left(\ref{eq:A2}\right)}}}{=} & \left(K\left(x_{i},x_{j}\right)\right)_{i,j=1}^{n};\label{eq:A5}
\end{eqnarray}
see (\ref{eq:d1}).

By the singular value decomposition, $L^{*}=WDU^{*}$, so that 
\begin{align}
LL^{*} & =UD^{2}U^{*},\label{eq:A6a}\\
L^{*}L & =WD^{2}W^{*},\label{eq:A6b}
\end{align}
where $D=\text{diag}\left(\lambda_{j}\right)$ consists of the non-negative
eigenvalues of $\sqrt{LL^{*}}$. Therefore, 
\[
L^{*}=\sum\lambda_{j}\left|w_{j}\left\rangle \right\langle u_{j}\right|.
\]

Note that $W=\left(w_{j}\right)$ is the KL basis that diagonalizes
$L^{*}L$ as in (\ref{eq:dp2}), i.e., 
\begin{equation}
L^{*}L=\sum_{j=1}^{n}\lambda_{j}^{2}\left|w_{j}\left\rangle \right\langle w_{j}\right|.
\end{equation}
It also follows from (\ref{eq:A6a})--(\ref{eq:A6b}), that 
\begin{equation}
W=L^{*}UD^{-1}.\label{eq:A8}
\end{equation}

\begin{rem}
\label{rem:In-the-above}In the above discussion, $\Phi\left(X\right)$
may be centered by removing its mean. Specifically, let 
\[
J=1-\frac{1}{n}\left|\mathbbm{1}\left\rangle \right\langle \mathbbm{1}\right|
\]
be the projection onto $\text{span}\left\{ \mathbbm{1}\right\} ^{\perp}$,
where $\mathbbm{1}$ denotes the constant vector $\begin{bmatrix}1 & \cdots & 1\end{bmatrix}$.
Then 
\begin{equation}
\widetilde{\Phi}\left(X\right):=\Phi\left(X\right)-\frac{1}{n}\sum_{j=1}^{n}\Phi\left(x_{j}\right)=\Phi\left(X\right)J,\label{eq:pc1}
\end{equation}
and so 
\begin{equation}
LL^{*}=\widetilde{\Phi}\left(X\right)^{*}\widetilde{\Phi}\left(X\right)=J\Phi\left(X\right)^{*}\Phi\left(X\right)J.\label{eq:pc2}
\end{equation}
The effect of $J$ in (\ref{eq:pc2}) is to exclude the eigenspace
of the Gramian $\left(K\left(x_{i},x_{j}\right)\right)_{i,j=1}^{n}$
spanned by the constant eigenvector.

In what follows, we shall always assume $\Phi\left(X\right)$ is centered
as in (\ref{eq:pc1})-(\ref{eq:pc2}). 
\end{rem}

\subsection{\label{subsec:FS}Feature Selection}

Feature selection, also called variable selection, or attribute selection,
is a procedure for automatic selection of those attributes in data
sets which are most relevant to particular predictive modeling problems.
Which features should one use in designs of predictive models? This
is a difficult question that requires detailed knowledge of the problem
at hand. The aim is algorithmic designs which automatically select
those features from prescribed data, which are most useful, or most
relevant, for the particular problem. The process is called feature
selection. A central premise of feature selection is that the input
data will contain features that are either redundant or irrelevant,
and can therefore be removed. The use of sample correlations in the
process is based in turn on the following principle: A particular
relevant feature might be redundant, in the presence of some other
relevant feature, with which it is strongly correlated.

Our present purpose is not a systematic treatment of feature selection,
but merely to identify how our present tools suggest recursive algorithms
in the general area. With this in mind we now consider the following
setup:

Let $x\in\mathbb{C}^{m}$ be a test example. The image $\Phi\left(x\right)$
under the feature map can be projected onto the principal directions
in $\mathcal{H}\left(K\right)$, via 
\[
\Phi\left(x\right)\longmapsto WW^{*}\Phi\left(x\right).
\]
The mapping $x\mapsto WW^{*}\Phi\left(x\right)$ is in general nonlinear.
See Examples \ref{exa:sc} and \ref{exa:dr}, and Figures \ref{fig:bc}
and \ref{fig:bc-1} below. We now turn to our main result which is
the following corollary.
\begin{cor}
\label{cor:pce}Let $L^{*}=WDU^{*}$ be as above, assuming $D$ is
full rank. For all $x\in\mathbb{C}^{m}$, the coefficients of the
projection $WW^{*}\Phi\left(x\right)$ are 
\begin{equation}
W^{*}\Phi\left(x\right)=\begin{bmatrix}\lambda_{1}^{-1}\sum_{j=1}^{n}\overline{u_{j1}}K\left(x_{j},x\right)\\
\vdots\\
\lambda_{n}^{-1}\sum_{j=1}^{n}\overline{u_{jn}}K\left(x_{j},x\right)
\end{bmatrix}.\label{eq:bc1}
\end{equation}
\end{cor}

\begin{proof}
By (\ref{eq:A8}), $W^{*}=D^{-1}U^{*}L$, so that 
\begin{eqnarray*}
W^{*}\Phi\left(x\right) & = & D^{-1}U^{*}L\Phi\left(x\right)\\
 & \underset{\left(\ref{eq:A4}\right)}{=} & D^{-1}U^{*}\begin{bmatrix}\left\langle \Phi\left(x_{1}\right),\Phi\left(x\right)\right\rangle _{\mathcal{H}\left(K\right)}\\
\vdots\\
\left\langle \Phi\left(x_{n}\right),\Phi\left(x\right)\right\rangle _{\mathcal{H}\left(K\right)}
\end{bmatrix}\\
 & \underset{\left(\ref{eq:A2}\right)}{=} & D^{-1}U^{*}\begin{bmatrix}K\left(x_{1},x\right)\\
\vdots\\
K\left(x_{n},x\right)
\end{bmatrix}\\
 & = & \begin{bmatrix}\lambda_{1}^{-1}\sum_{j=1}^{n}\overline{u_{j1}}K\left(x_{j},x\right)\\
\vdots\\
\lambda_{n}^{-1}\sum_{j=1}^{n}\overline{u_{jn}}K\left(x_{j},x\right)
\end{bmatrix}.
\end{eqnarray*}
\end{proof}
\begin{example}[Spectral Clustering, see \figref{bc}]
\label{exa:sc} This is included as an instructive example. The data
set $X$ has two classes: Class ``0'' consists of $867$ points
uniformly distributed in $\{x\in\mathbb{R}^{3}\mathrel{;}0.6<\left\Vert x\right\Vert ^{2}<1\}$;
class ``1'' consists of 126 points in the open ball $\{x\in\mathbb{R}^{3}\mathrel{;}\left\Vert x\right\Vert ^{2}<0.2\}$.
Hence $X$ has dimension $3\times993$, and each column $x_{j}$ of
$X$ corresponds to a sample point.

We choose the Gaussian kernel $K\left(x,y\right)=e^{-\left\Vert x-y\right\Vert ^{2}/\sigma}$
with $\sigma=0.05$, so that $x_{j}$ is embedded into the associated
RKHS by 
\[
\mathbb{R}^{3}\ni x_{j}\longmapsto\Phi\left(x_{j}\right)=e^{-\left\Vert \cdot-x_{j}\right\Vert ^{2}/\sigma}\in\mathcal{H}\left(K\right).
\]
By projecting $\Phi\left(x_{j}\right)$ onto the first two principal
components, each sample point $x_{j}\in\mathbb{R}^{3}$ has a 2D representation
via the mapping 
\[
\mathbb{R}^{3}\ni x_{j}\longmapsto\Phi\left(x_{j}\right)\longmapsto\begin{bmatrix}\lambda_{1}^{-1}\sum_{k=1}^{n}\overline{u_{k1}}K\left(x_{k},x_{j}\right)\\
\lambda_{2}^{-1}\sum_{k=1}^{n}\overline{u_{k2}}K\left(x_{k},x_{j}\right)
\end{bmatrix}\in\mathbb{R}^{2}.
\]
See the general formula in \corref{pce}; note that $n=2$ in the
present example.

As shown in \figref{bc2}, the two clusters are linearly separable
in $\mathcal{H}\left(K\right)$. 
\end{example}

\begin{figure}
\subfloat[]{\includegraphics[width=0.4\columnwidth]{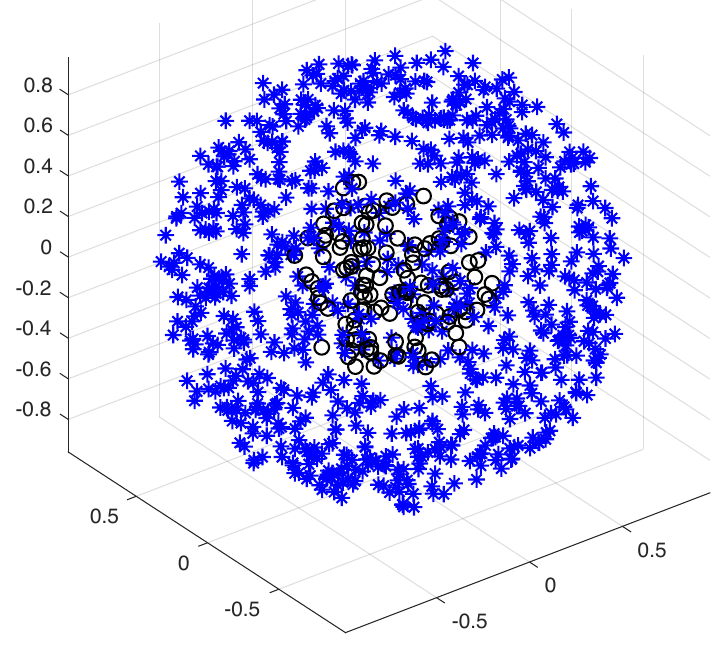}

}\hfill{}\subfloat[\label{fig:bc2}]{\includegraphics[width=0.4\columnwidth]{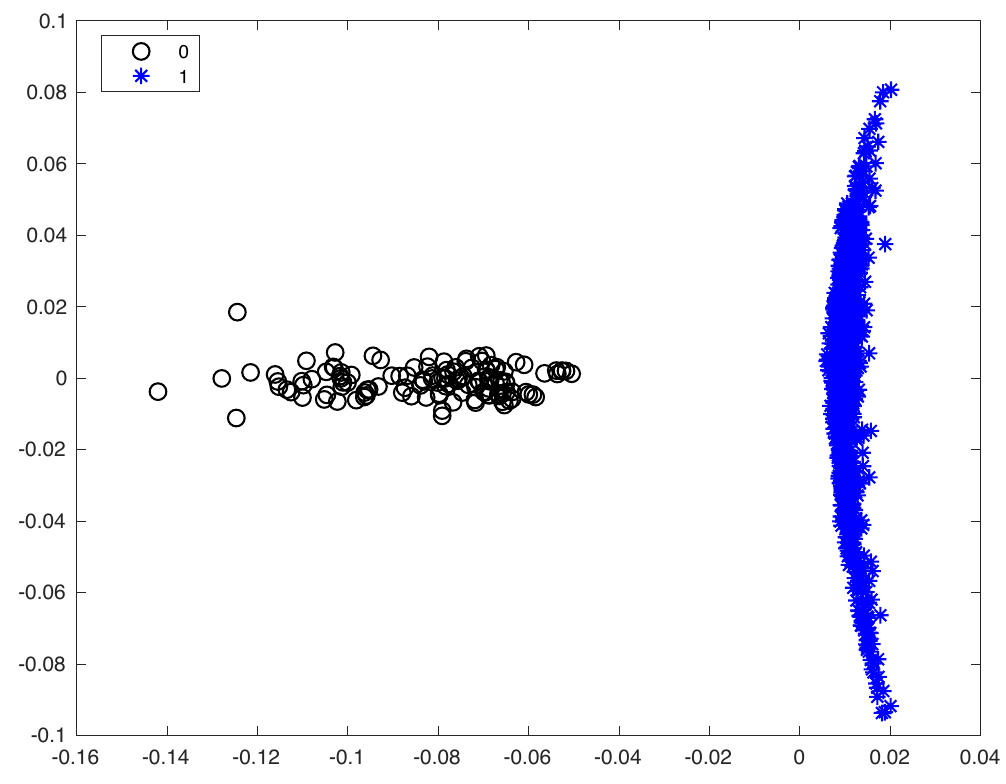}

}

\caption[Spectral clustering and linear decision boundary via Gaussian kernels
and KPCA.]{\label{fig:bc}Spectral clustering via Gaussian kernel. (A) The two
classes \textquotedblleft$\ast$\textquotedblright , and \textquotedblleft o\textquotedblright{}
are not separable by standard PCA. (B) KPCA with Gaussian kernel $K\left(x,y\right)=e^{-\left\Vert x-y\right\Vert ^{2}/\sigma}$,
$\sigma=0.05$. Project $\Phi\left(X\right)$ onto the first two principal
components, then the resulting 2D representation is separable by a
linear decision boundary in $\mathcal{H}\left(K\right)$. \protect \\
This figure is an illustration of the feature selection in \corref{pce}.
For the present binary clustering problem (\textquotedblleft$*$\textquotedblright{}
vs \textquotedblleft o\textquotedblright ), the kernel $K$ is centered
using (\ref{eq:pc1})--(\ref{eq:pc2}), and the vector $\Phi\left(X\right)$
with entries in the Gaussian RKHS $\mathcal{H}\left(K\right)$ is
projected onto the first two column vectors of the matrix $W$ using
(\ref{eq:bc1}). For a fixed sample point $x$, the two projection
coefficients are $\lambda_{1}^{-1}\sum_{j=1}^{n}\overline{u_{j1}}K\left(x_{j},x\right)$
and $\lambda_{2}^{-1}\sum_{j=1}^{n}\overline{u_{j2}}K\left(x_{2},x\right)$.}
\end{figure}

\begin{example}[Dimension Reduction, see \figref{bc-1}]
\label{exa:dr} Let $X$ be a collection of 100 grayscale $256\times256$
images of an ellipse, rotated successively by $\pi/100$. \figref{krot1}
shows 6 sample images corresponding to different rotation angles.
The images are unrolled as column vectors, thus $X$ has dimension
$65536\times100$.

This data set may be viewed as 1D submanifold embedded in $\mathbb{R}^{65536}$,
i.e., it has only one degree of freedom, the rotation angle. For dimension
reduction, KPCA will ideally extract this information, and each image
is then represented by a single projection coefficient. We choose
the Gaussian kernel with $\sigma=300$. In \figref{krot2}, there
are 4 subplots consisting of the projections onto the first, second,
third, and fourth principal directions. The rotation angle is encoded
in e.g. PC 1. As a consequence, the dimension of the data set is reduced
from $65536\times100$ to $1\times100$.

In particular, the projection coefficients onto the $k^{th}$ principal
direction (see \figref{krot2}, PC1 -- PC4) are proportional to the
$k^{th}$ eigenvector of the centered Gramian $G:=JK\left(x_{i},x_{j}\right)J$;
see (\ref{eq:pc2}) and \corref{pce}. 
\end{example}

\begin{figure}
\subfloat[\label{fig:krot1}]{\includegraphics[width=0.6\columnwidth]{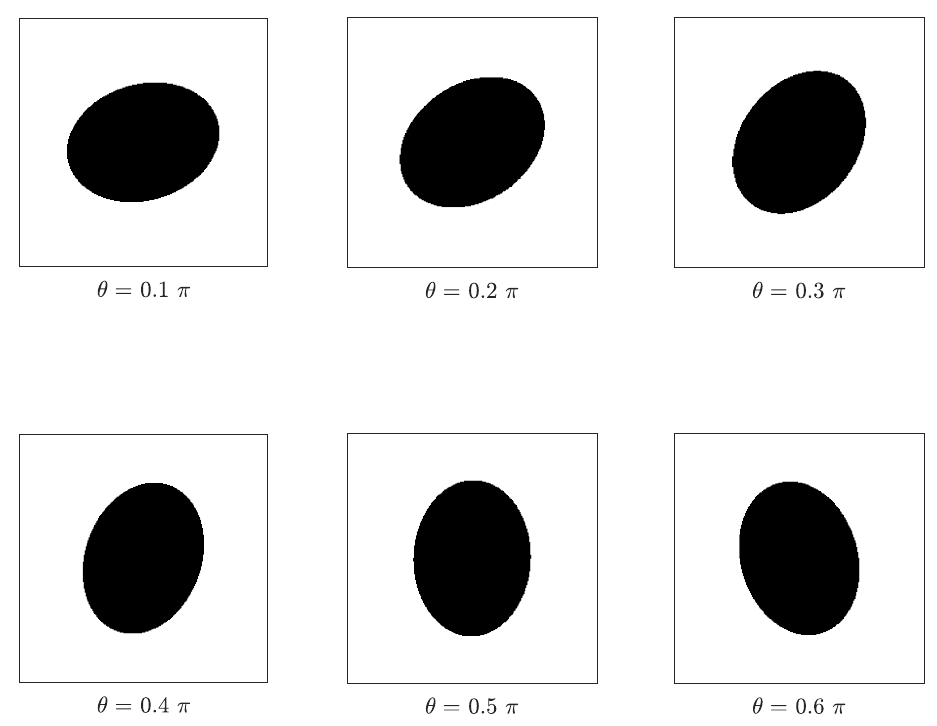}

}

\subfloat[\label{fig:krot2}]{\includegraphics[width=0.6\columnwidth]{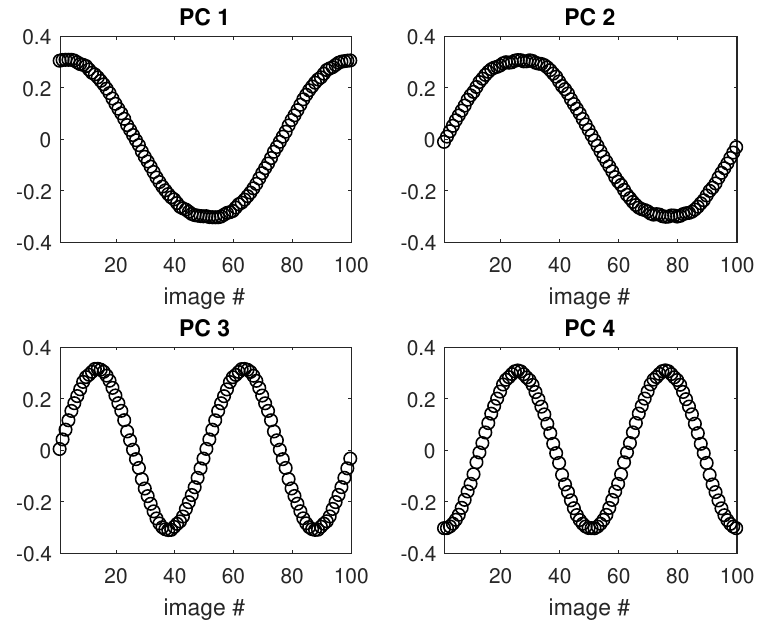}

}

\caption[Non-linear data, and detection of rotation angles via KPCA.]{\label{fig:bc-1}Detection of rotation angle by KPCA. (A) Sample
points from the data set $X$ of 100 grayscale images of an ellipse,
rotated successively by $\pi/100$, each has resolution $256\times256$.
(B) Apply kernel PCA with Gaussian kernel $K\left(x,y\right)=e^{-\left\Vert x-y\right\Vert ^{2}/\sigma}$,
$\sigma=300$, then project $\Phi\left(X\right)$ onto the first,
second, third, and forth principal directions in $\mathcal{H}\left(K\right)$.
The rotation angle is captured in e.g. PC 1.\protect \\
This figure illustrates KPCA feature selection using \corref{pce}.
The Gaussian kernel is centered by (\ref{eq:pc1})--(\ref{eq:pc2}).
For each fixed sample point $x$, the vector $\Phi\left(x\right)$
in the Gaussian RKHS $\mathcal{H}\left(K\right)$ is projected onto
the first 4 column vectors of the matrix $W$ as in (\ref{eq:bc1}).
The resulting projection coefficients are given by $\lambda_{1}^{-1}\sum_{k=1}^{n}\overline{u_{k1}}K\left(x_{k},x\right)$,
$\lambda_{2}^{-1}\sum_{k=1}^{n}\overline{u_{k2}}K\left(x_{k},x\right)$,
$\lambda_{3}^{-1}\sum_{k=1}^{n}\overline{u_{k3}}K\left(x_{k},x\right)$
and $\lambda_{4}^{-1}\sum_{k=1}^{n}\overline{u_{k4}}K\left(x_{k},x\right)$.}
\end{figure}

\subsection{\label{subsec:dyn}Dynamic PCA}

In some applications, the tool of \emph{Dynamic PCA} (DPCA), see e.g.,
\cite{MR2314342} serves as a useful tool for \textbf{high-dimensional
and time-dependent data}. The idea is that in favorable cases, with
the use of DPCA, one can arrange that the input matrices can be augmented
by addition of time-lagged values of the variables under consideration.
In summary, the method is based on the behavior of the eigenvalues
of the lagged autocorrelation, and partial autocorrelation, matrices.

Our present \thmref{rcpt} states that when a system of PCA eigenvalues
$\left\{ \lambda_{i}\right\} $ is given, then we have a algorithmic
solution to the corresponding optimization question (see (\ref{eq:q1})-(\ref{eq:q2})).
We now turn to a formula for generating PCA features as a limit of
a certain iteration of operators. The family of operators $T$ discussed
below is a generalization of the operators from \subsecref{of} above.

When principal component analysis (PCA) is used for problems involving
stochastic (statistical) processes, for example in monitoring applications,
it often relies on an implicit, but unrealistic, assumption that data
involved are time independent. This however is evidently unrealistic,
as work with industrial data shows: In a host of diverse applications,
one is faced with serial correlations. Hence, the literature over
recent years (see e.g., \cite{MR4188895,MR4131352,MR4097183,MR4085822,MR4010088})
has witnessed a host of new approaches to PCA, going by what is now
called dynamic PCA (DPCA) methods. DPCA-models involve a variety of
stochastic and dynamical systems-tools, each one serving as a remedy
for plain-vanilla PCA applied to data as they appear in: (i) in extremely
high-dimensions, (ii) in time-dependent data-sets, and (iii) in digital
image processing. A main feature of DPCA is time-feedback; --- hence
in DPCA, the input-matrix is dynamically augmented, thus taking the
form of a transfer operator;--- in the simplest models, it operates
by dynamic additions of time-lagged values of the variables. 
\begin{prop}
\label{prop:pei}Suppose $T:\mathcal{H}\rightarrow\mathcal{H}$ is
compact, and $T^{*}T$ has simple spectrum, i.e., 
\[
T^{*}T=\sum_{j=1}^{\infty}\lambda_{j}^{2}\left|u_{j}\left\rangle \right\langle u_{j}\right|
\]
with $\lambda_{1}^{2}>\lambda_{2}^{2}>\cdots>0$; $\lambda_{j}^{2}\rightarrow0$
as $j\rightarrow\infty$. Setting $S=\sqrt{T^{*}T}$, then, 
\begin{equation}
\lim_{n\rightarrow\infty}\frac{\left\Vert S^{n+1}x\right\Vert ^{2}}{\left\Vert S^{n}x\right\Vert ^{2}}=\lambda_{1}^{2},\quad\forall x\;\text{s.t. }\left\langle u_{1},x\right\rangle \neq0.
\end{equation}

Moreover, given $\lambda_{1}$, then 
\begin{equation}
\lim_{n\rightarrow\infty}\lambda_{1}^{-2n}\left(T^{*}T\right)^{n}=\left|u_{1}\left\rangle \right\langle u_{1}\right|,\label{eq:p2}
\end{equation}
where convergence in (\ref{eq:p2}) is w.r.t. the norm topology of
$B\left(\mathcal{H}\right)$. 
\end{prop}

\begin{proof}
Note that 
\begin{align*}
\lim_{n\rightarrow\infty}\frac{\left\Vert S^{n+1}x\right\Vert ^{2}}{\left\Vert S^{n}x\right\Vert ^{2}} & =\lim_{n\rightarrow\infty}\frac{\left\langle x,S^{2n+2}x\right\rangle }{\left\langle x,S^{2n}x\right\rangle }\\
 & =\lim_{n\rightarrow\infty}\frac{\left\langle x,\left(T^{*}T\right)^{n+1}x\right\rangle }{\left\langle x,\left(T^{*}T\right)^{n}x\right\rangle }\\
 & =\lim_{n\rightarrow\infty}\frac{\sum\lambda_{j}^{2\left(n+1\right)}\left|\left\langle u_{j},x\right\rangle \right|^{2}}{\sum\lambda_{j}^{2n}\left|\left\langle u_{j},x\right\rangle \right|^{2}}\\
 & =\lim_{n\rightarrow\infty}\lambda_{1}^{2}\frac{\left|\left\langle u_{1},x\right\rangle \right|^{2}+\sum_{j\geq2}\left(\lambda_{j}/\lambda_{1}\right)^{2\left(n+1\right)}\left|\left\langle u_{j},x\right\rangle \right|^{2}}{\left|\left\langle u_{1},x\right\rangle \right|^{2}+\sum_{j\geq2}\left(\lambda_{j}/\lambda_{1}\right)^{2n}\left|\left\langle u_{j},x\right\rangle \right|^{2}}\\
 & =\lambda_{1}^{2}.
\end{align*}

Now, given $\lambda_{1}$, we have 
\begin{align*}
\left\Vert \lambda_{1}^{-2n}\left(T^{*}T\right)^{n}-\left|u_{1}\left\rangle \right\langle u_{1}\right|\right\Vert  & =\left\Vert \sum\nolimits _{j\geq2}\left(\lambda_{j}/\lambda_{1}\right)^{2n}\left|u_{j}\left\rangle \right\langle u_{j}\right|\right\Vert \\
 & \leq\sum\nolimits _{j\geq2}\left(\lambda_{j}/\lambda_{1}\right)^{2n}\xrightarrow[\;n\rightarrow\infty\;]{}0.
\end{align*}
\end{proof}
\begin{cor}
\label{cor:The-system-of}The system of PCA eigenvalues $\left\{ \lambda_{j}\right\} $
can be obtained inductively as follows: Let $S=\sqrt{T^{*}T}$ as
before, and set $Q_{k}=\sum_{j=1}^{k}\left|u_{j}\left\rangle \right\langle u_{j}\right|$.
Then 
\[
\lambda_{k+1}^{2}=\lim_{n\rightarrow\infty}\frac{\left\Vert \left(SQ_{k}^{\perp}\right)^{n+1}x\right\Vert ^{2}}{\left\Vert \left(SQ_{k}^{\perp}\right)^{n}x\right\Vert },\quad\forall x\;\text{s.t.}\;\left\langle u_{k+1},x\right\rangle \neq0,
\]
and 
\[
\left|u_{k+1}\left\rangle \right\langle u_{k+1}\right|=\lim_{n\rightarrow\infty}\lambda_{k}^{-2n}\left(Q_{k}^{\perp}T^{*}TQ_{k}^{\perp}\right)^{n}.
\]
\end{cor}

\begin{proof}
One checks that 
\begin{align*}
SQ_{k}^{\perp} & =S\left(1-Q_{k}\right)=\sum_{j=k+1}^{\infty}\lambda_{j}\left|u_{j}\left\rangle \right\langle u_{j}\right|,\\
\left(SQ_{k}^{\perp}\right)^{*} & \left(SQ_{k}^{\perp}\right)=Q_{k}^{\perp}T^{*}TQ_{k}^{\perp};
\end{align*}
and so the assertion follows from \propref{pei}. 
\end{proof}
\begin{example}
\label{exa:bm}Consider the covariance function of standard Brownian
motion $B_{t}$, $t\in[0,\infty)$, i.e., a Gaussian process $\left\{ B_{t}\right\} $
with mean zero and covariance function 
\begin{equation}
\mathbb{E}\left(B_{s}B_{t}\right)=s\wedge t=\min\left(s,t\right).\label{eq:bm1}
\end{equation}
Let $F_{N}=\left\{ x_{1},x_{2},\ldots,x_{N}\right\} $ be a finite
subsets of $V$, such that 
\[
0<x_{1}<x_{2}<\cdots<x_{N};
\]
and let 
\begin{equation}
K_{N}=\begin{bmatrix}x_{1} & x_{1} & x_{1} & \cdots & x_{1}\\
x_{1} & x_{2} & x_{2} & \cdots & x_{2}\\
x_{1} & x_{2} & x_{3} & \cdots & x_{3}\\
\vdots & \vdots & \vdots & \vdots & \vdots\\
x_{1} & x_{2} & x_{3} & \cdots & x_{N}
\end{bmatrix}=\left(x_{i}\wedge x_{j}\right)_{i,j=1}^{N}.\label{eq:bm4}
\end{equation}
\end{example}

\begin{lem}
Let $K_{N}$ be as in (\ref{eq:bm4}). Then 
\begin{enumerate}
\item The determinant of $K_{N}$ is given by 
\begin{equation}
\det\left(K_{N}\right)=x_{1}\left(x_{2}-x_{1}\right)\left(x_{3}-x_{2}\right)\cdots\left(x_{N}-x_{N-1}\right).\label{eq:bm5}
\end{equation}
\item $K_{N}$ assumes the LU decomposition 
\begin{equation}
K_{N}=A_{N}A_{N}^{*},
\end{equation}
where 
\begin{equation}
A_{n}=\begin{bmatrix}\sqrt{x_{1}} & 0 & 0 & \cdots & 0\\
\sqrt{x_{1}} & \sqrt{x_{2}-x_{1}} & 0 & \cdots & \vdots\\
\sqrt{x_{1}} & \sqrt{x_{2}-x_{1}} & \sqrt{x_{3}-x_{2}} & \ddots & \vdots\\
\vdots & \vdots & \vdots & \ddots & 0\\
\sqrt{x_{1}} & \sqrt{x_{2}-x_{1}} & \sqrt{x_{3}-x_{2}} & \cdots & \sqrt{x_{N}-x_{N-1}}
\end{bmatrix}.\label{eq:H31}
\end{equation}
\end{enumerate}
\end{lem}

\begin{proof}
For details, see e.g., \cite{MR3450534}. 
\end{proof}
\begin{example}
Fix $N$, then the top eigenvalue of $K:=K_{N}$ can be extracted
by the method from \propref{pei}. For instance, if $N=3$ and $F_{3}=\left\{ 1,2,3\right\} $,
we have 
\[
K=\begin{bmatrix}1 & 1 & 1\\
1 & 2 & 2\\
1 & 2 & 3
\end{bmatrix}.
\]
Let $e_{1}=\left[1,0,0\right]$, then 
\[
\hat{\lambda}_{1}=\frac{\left\Vert K^{3}e_{1}\right\Vert }{\left\Vert K^{2}e_{1}\right\Vert }\approx5.0455,\quad\hat{v}_{1}=\frac{\lambda_{1}^{-5}K^{5}e_{1}}{\left\Vert \lambda_{1}^{-5}K^{5}e_{1}\right\Vert }\approx\begin{pmatrix}0.3284 & 0.5913 & 0.7366\end{pmatrix}^{T}.
\]
Standard numerical algorithm returns 
\[
\lambda_{1}\approx5.0489,\quad v_{1}\approx\begin{pmatrix}0.3280 & 0.5910 & 0.7370\end{pmatrix}^{T}.
\]
\end{example}

\begin{acknowledgement*}
The co-authors thank the following colleagues for helpful and enlightening
discussions: Professors Daniel Alpay, Sergii Bezuglyi, Ilwoo Cho,
Wayne Polyzou, David Stewart, Eric S. Weber, and members in the Math
Physics seminar at The University of Iowa. The present work was started
with discussions at the NSF-CBMS Conference, Harmonic Analysis: Smooth
and Non-Smooth, held at Iowa State University, June 4-8, 2018. Jorgensen
was the main speaker. We are grateful to the organizers, especially
to Prof Eric Weber, and to the NSF for financial support. Sooran Kang
was supported by Basic Science Research Program through the National
Research Foundation of Korea (NRF) funded by the Ministry of Education
(\#NRF-2017R1D1A1B03034697 and \#NRF-2020R1F1A1A01076072). 
\end{acknowledgement*}
\bibliographystyle{amsalpha}
\bibliography{encode}
 
\end{document}